\documentclass[a4paper,oneside,11pt,reqno]{article}%

\usepackage[T1]{fontenc}
\usepackage[english]{babel}
\usepackage{amsmath,amssymb,amsthm}
\usepackage[margin=3.7cm]{geometry}
\usepackage{esint}
\usepackage{eucal}
\usepackage[font=small]{caption}
\usepackage[caption=false]{subfig}
\usepackage{graphicx}
\usepackage{listings}
\usepackage{tikz}
\usetikzlibrary{arrows,decorations.markings}
 \tikzstyle{vecArrow} = [thick, decoration={markings,mark=at position
   1 with {\arrow[semithick]{open triangle 60}}},
   double distance=2.5pt, shorten >= 5.5pt, shorten <= .5pt,
   preaction = {decorate}]

\usepackage{algpseudocode}

\usepackage{fancyhdr}
 \newtheorem{theorem}{Theorem}
 \newtheorem{lemma}[theorem]{Lemma}
 \newtheorem{proposition}[theorem]{Proposition}
 
\newtheorem{assumption}[theorem]{Assumption}
 \newtheorem{remark}[theorem]{Remark}
\theoremstyle{definition}
\newtheorem{algo}[theorem]{Algorithm}

 \numberwithin{theorem}{section}
 \numberwithin{equation}{section}

 \newcommand{\ddiv}{\operatorname{div}}
 \newcommand{\tri}{\mathcal{T}}
 \newcommand{\nodes}{\mathcal{N}}
 \newcommand{\edges}{\mathcal{E}}
 \newcommand{\Curl}{\operatorname{Curl}}
 \newcommand{\curl}{\operatorname{curl}}

 \newcommand{\R}{\mathbb{R}}

 \newcommand{\NC}{{\text{\!\tiny \rm N\!C}}}

\begin{document}

\author{M. Schedensack\thanks{Institut f\"ur Numerische Simulation, 
            Universit\"at Bonn, 
            Wegelerstra{\ss}e 6, D-53115 Bonn, Germany
             }
} 
\title{A new generalization of the $P_1$ non-conforming FEM to higher polynomial degrees\thanks{%
        This work was supported by the Berlin Mathematical School.}
          }
\date{}
 \maketitle
 
\begin{abstract}
This paper generalizes the non-conforming FEM of Crouzeix and Raviart 
and its fundamental projection property
by a novel mixed formulation for the 
Poisson problem based on the Helm\-holtz decomposition.
The new formulation allows for ansatz spaces of arbitrary polynomial degree and 
its discretization coincides with the mentioned non-conforming FEM for the 
lowest polynomial degree. 
The discretization directly approximates the gradient of the solution 
instead of the solution itself.
Besides the a~priori and medius analysis, this paper proves optimal 
convergence rates for an adaptive algorithm for the new discretization.
These are also demonstrated in numerical experiments.
Furthermore, this paper focuses on extensions of this new scheme to 
quadrilateral meshes, mixed FEMs, and three space dimensions.
\end{abstract}
 


\section{Introduction}

Non-conforming finite element methods (FEMs) play an important role in 
computational mechanics. They allow the discretization of partial differential 
equations (PDEs) for 
incompressible fluid flows, for almost 
incompressible materials in linear elasticity, and for low polynomial degrees 
in the ansatz spaces for higher-order problems. 
The projection property of the interpolation operator of the 
$P_1$ non-conforming FEM, also named after Crouzeix and 
Raviart~\cite{CrouzeixRaviart1973}, states that the $L^2$~projection of 
$\nabla H^1_0(\Omega)$ onto the space of piecewise constant functions equals 
the space of piecewise gradients of the non-conforming interpolation of 
$H^1_0(\Omega)$~functions in the $P_1$ non-conforming finite element space.
This property is the basis for the proof of the discrete inf-sup condition 
for the Stokes equations~\cite{CrouzeixRaviart1973} as well as for the 
analysis of adaptive algorithms~\cite{BeckerMaoShi2010}.

Many possible generalizations of the $P_1$ non-conforming FEM to 
higher polynomial degrees have been proposed. All those generalizations 
are either based on a 
modification of the classical concept of degrees of freedom 
\cite{FortinSoulie1983,Fortin1985,StoyanBaran2006},
are restricted to odd polynomial degrees
\cite{CrouzeixFalk1989,ArnoldBrezzi1985},
or employ an enrichment by additional bubble-functions
\cite{MaubachRabier2003,MatthiesTobiska2005}.
However, none of those generalizations possesses a corresponding  
projection property of the interpolation operator for higher moments 
(see Remark~\ref{r:PMPnoteqCR} below).
This paper introduces a novel formulation of the Poisson equation 
(in~\eqref{e:PMPmixedproblem} below) 
based on the Helmholtz decomposition
along with its discretization of arbitrary 
(globally fixed) polynomial degree.
This new discretization approximates directly the gradient of the 
solution, which is often the quantity of interest, instead of the solution 
itself.
For the lowest-order polynomial degree, the discrete 
Helmholtz decomposition of \cite{ArnoldFalk1989} proves equivalence of the 
novel discretization with the known
non-conforming Crouzeix-Raviart FEM \cite{CrouzeixRaviart1973}
and therefore they appear in a natural hierarchy. 
In the context of the novel (mixed) formulation, these discretizations 
turn out to be conforming. 
Although the complexity of the new discretization itself is competitive 
with that of a standard FEM, the method requires the pre-computation of some 
function $\varphi$ such that its divergence equals the right-hand side. If 
this is not computable analytically, this results in an additional 
integration (see also Remark~\ref{r:computationphi} below).
However, this paper focuses on the Poisson problem as a model problem 
to introduce the idea of the new approach and to give a broad impression 
over possible extensions as quadrilateral discretizations 
(including a discrete Helmholtz decomposition on quadrilateral meshes 
for the non-conforming Rannacher-Turek FEM~\cite{RannacherTurek1990} as
a further highlight of this paper), the generalization to three dimensions, 
or inhomogeneous mixed boundary conditions. 
The advantages of the new approach in some 
applications will be the topic of forthcoming 
papers~\cite{SchedensackmLapl2015, SchedensackLE2016}.

The presence of singularities for non-convex domains usually yields the same 
sub-optimal convergence rate for any polynomial degree. This motivates adaptive 
mesh-generation strategies, which recover the optimal convergence rates.
This paper presents an adaptive algorithm and proves its optimal convergence. 
The proof essentially follows ideas from the 
context of the non-conforming Crouzeix-Raviart FEM 
\cite{BeckerMaoShi2010,Rabus2010}. This illustrates 
that the novel discretization generalizes it in a natural way.
Since the 
efficient and reliable error estimator involves a data approximation term 
without a multiplicative power of the mesh-size, the adaptive algorithm is 
based on separate marking.

A possible drawback of the new FEMs is that the gradient of the solution $\nabla u$ is 
approximated, but not the solution $u$ itself. This excludes obvious 
generalizations 
to partial differential equations where $u$ appears in lower-order terms.

The remaining parts of this paper are organized as follows. 
Section~\ref{s:notation} defines some notation.
Section~\ref{s:PMPformulation} introduces the novel formulation 
based on the Helmholtz decomposition and its discretization 
together with an a~priori error estimate.
The equivalence with the $P_1$ non-conforming FEM for the 
lowest-order case is proved in Subsection~\ref{ss:PMPCR}.
Section~\ref{s:PMPremarks} summarizes some generalizations.
Section~\ref{s:PMPmedius} is devoted to a medius analysis 
of the FEM, which uses a~posteriori techniques to derive a~priori error 
estimates. Section~\ref{s:PMPafem} proves quasi-optimality 
of an adaptive algorithm, while Section~\ref{s:PMP3D} outlines the generalization 
to 3D. Section~\ref{s:PMPnumerics} concludes this 
paper with numerical experiments.

\section{Notation}\label{s:notation}

Throughout this paper
$\Omega\subseteq\R^2$ is a simply connected, bounded, polygonal Lip\-schitz 
domain.
Standard notation on Lebesgue and Sobolev
spaces and their norms is employed with $L^2$ scalar product 
$(\bullet,\bullet)_{L^2(\Omega)}$. Given a Hilbert space $X$, 
let $L^2(\Omega;X)$ resp.\ $H^k(\Omega;X)$ denote the space of functions with 
values in $X$ whose components are in $L^2(\Omega)$ resp.\ $H^k(\Omega)$
and let $L^2_0(\Omega)$ denote the subset of $L^2(\Omega)$ of functions with 
vanishing integral mean.
The space of $L^2$ functions whose weak divergence exists and is in $L^2$ is 
denoted with $H(\ddiv,\Omega)$.
The space of infinitely differentiable 
functions reads $C^\infty(\Omega)$ and the subspace of functions 
with compact support in $\Omega$ is denoted with $C^\infty_c(\Omega)$.
The piecewise action of differential operators
is denoted with a subscript $\NC$.
The formula $A\lesssim B$ represents an inequality $A\leq CB$ 
for some mesh-size independent, positive generic constant
$C$; $A \approx B$ abbreviates $A \lesssim B \lesssim A$.
By convention, all generic constants $C\approx 1$ do neither
depend on the mesh-size nor on the level of a triangulation
but may depend on the fixed coarse 
triangulation $\tri_0$ and its interior angles.
The Curl operator in two dimensions is defined by 
$\Curl\beta:=(\partial \beta/\partial x_2, -\partial \beta/\partial x_1)$ for 
sufficiently smooth $\beta$.

A shape-regular triangulation $\tri$ of a bounded, polygonal, open
Lipschitz domain $\Omega\subseteq \mathbb{R}^2$ is a set of closed triangles $T\in \tri$
such that $\overline{\Omega}=\bigcup\tri$ and any two distinct 
triangles are either disjoint or share exactly 
one common edge or one vertex.
Let $\edges(T)$ denote the edges of a triangle $T$
and $\edges:=\edges(\tri):=\bigcup_{T\in\tri} \edges(T)$ the set of edges 
in $\tri$. 
Any edge $E\in\edges$ is associated with a fixed orientation of the unit normal 
$\nu_E$ on $E$ (and $\tau_E=(0,-1;1,0)\nu_E$ denotes the unit tangent on $E$). 
On the boundary, $\nu_E$ is the outer unit normal of $\Omega$,
while for interior edges $E\not\subseteq\partial\Omega$, the orientation is 
fixed through the choice of the triangles $T_+\in\tri$ and $T_-\in\tri$ with 
$E=T_+\cap T_-$ and $\nu_E:=\nu_{T_+}\vert_E$ is then the outer normal of $T_+$ on $E$.
In this situation, $[v]_E:=v\vert_{T_+}-v\vert_{T_-}$ denotes the jump across 
$E$. For an edge $E\subseteq\partial\Omega$ on the boundary, the jump across 
$E$ reads $[v]_E:=v$.
For $T\in\tri$ and $X\subseteq \R^n$, let 
\begin{align*}
P_k(T;X) &:= \left\{v:T\rightarrow X \left\vert 
     \begin{array}{l}
          \text{each component}\text{ of }v
    \text{ is a polynomial}\\
     \text{of total degree}\leq k
     \end{array}
   \right\}\right.;\\
P_k(\tri;X) &:= \{v:\Omega\rightarrow X\;| \;
\forall T\in\tri:\;v|_T \in P_k(T;X)\}
\end{align*}
denote the set of piecewise polynomials and $P_k(\tri):=P_k(\tri;\R)$. 
Given a subspace $X\subseteq L^2(\Omega;\R^n)$, 
let $\Pi_{X}:L^2(\Omega;\R^n)\to X$
denote the $L^2$ projection onto $X$ and let $\Pi_k$ abbreviate 
$\Pi_{P_k(\tri;\R^n)}$.
Given a triangle $T\in\tri$, let $h_T:=(\mathrm{meas}_2(T))^{1/2}$ 
denote the square root of the area of 
$T$ and 
let $h_\tri\in P_0(\tri)$ denote the piecewise constant mesh-size 
with $h_\tri\vert_T:=h_T$ for all $T\in\tri$. 
For a set of triangles $\mathcal{M}\subseteq\tri$,
let $\|\bullet\|_{\mathcal{M}}$ abbreviate 
\begin{align*}
 \|\bullet\|_{\mathcal{M}}:=\sqrt{\sum_{T\in\mathcal{M}}
         \|\bullet\|_{L^2(T)}^2}.
\end{align*}
Given an initial triangulation $\tri_0$, an admissible triangulation is 
a regular triangulation which can be created from $\tri_0$ by newest-vertex 
bisection \cite{Stevenson2008}. The set of admissible triangulations is 
denoted by $\mathbb{T}$.

\section{Problem formulation and discretization}\label{s:PMPformulation}

This section introduces the new formulation based on the 
Helmholtz decomposition in Subsection~\ref{ss:PMPformulation}
and its discretization in Subsection~\ref{ss:PMPdiscretisation}.
Subsection~\ref{ss:PMPCR} discusses the equivalence 
with the $P_1$ non-conforming Crouzeix-Raviart FEM \cite{CrouzeixRaviart1973}.

\subsection{New mixed formulation of the Poisson problem}\label{ss:PMPformulation}

Given the simply connected, bounded, polygonal Lipschitz domain 
$\Omega\subseteq \R^2$ and 
$f\in L^2(\Omega)$, the Poisson model problem seeks 
$u\in H^1_0(\Omega)$ with 
\begin{align}\label{e:PMPstrongform}
 -\Delta u = f \text{ in }\Omega
 \qquad\text{and}\qquad 
 u=0 \text{ on }\partial \Omega.
\end{align}
The novel weak formulation is based on the classical Helmholtz decomposition \cite{Rudin1976}
\begin{align}\label{e:HelmholtzDec}
 L^2(\Omega;\R^2) = \nabla H^1_0(\Omega)\oplus \Curl (H^1(\Omega)\cap 
L^2_0(\Omega))
\end{align}
for any simply connected domain $\Omega\subseteq\R^2$,
where the sum is orthogonal with respect to the 
$L^2$ scalar product.

\begin{remark}
Note that for $\Omega\subseteq\R^2$, the definition of the Curl implies 
\begin{equation*}
 H(\Curl,\Omega)
  :=\{\beta\in L^2(\Omega)\mid \Curl\beta\in L^2(\Omega)\}
  = H^1(\Omega).
 \qedhere
\end{equation*}
\end{remark}

Define $X:=L^2(\Omega;\R^2)$ and $Y:=H^1(\Omega)\cap L^2_0(\Omega)$ and 
let $\varphi\in H(\ddiv,\Omega)$ satisfy $-\ddiv\varphi = f$.
The novel weak formulation of the Poisson problem~\eqref{e:PMPstrongform}
seeks $(p,\alpha)\in X\times Y$ with
\begin{equation}\label{e:PMPmixedproblem}
\begin{aligned}
 (p,q)_{L^2(\Omega)} + (q,\Curl \alpha)_{L^2(\Omega)} &= 
(\varphi,q)_{L^2(\Omega)}
 &&\text{ for all }q\in X,\\
 (p,\Curl \beta)_{L^2(\Omega)} &= 0 
 &&\text{ for all }\beta\in Y.
\end{aligned}
\end{equation}
This formulation is the point of departure 
for the numerical approximation of $\nabla u$ in 
Subsection~\ref{ss:PMPdiscretisation}.

\begin{remark}[existence of solutions]\label{r:existence}
Since $\Curl Y\subseteq X$, any $\beta\in Y$ satisfies the inf-sup condition
\begin{align*}
 \left\|\Curl\beta\right\|_{L^2(\Omega)}
   \leq \sup_{q\in X\setminus\{0\}} 
       \frac{(q,\Curl\beta)_{L^2(\Omega)}}{\|q\|_{L^2(\Omega)}}.
\end{align*}
This and Brezzi's splitting lemma \cite{Brezzi1974} 
imply the unique existence of a solution $(p,\alpha)\in X\times Y$
to \eqref{e:PMPmixedproblem}.
The $L^2$ orthogonality of $p$ and $\Curl \alpha$ 
implies
\begin{align*}
 \|p\|_{L^2(\Omega)}^2 + \left\|\Curl\alpha\right\|_{L^2(\Omega)}^2 
  = \|\varphi\|_{L^2(\Omega)}^2.
\end{align*}
\end{remark}

\begin{remark}[equivalence of \eqref{e:PMPstrongform} and \eqref{e:PMPmixedproblem}]
The second equation of~\eqref{e:PMPmixedproblem} and the Helmholtz 
decomposition~\eqref{e:HelmholtzDec} imply the existence 
of $\widetilde{u}\in H^1_0(\Omega)$ with $p=\nabla\widetilde{u}$.
Since $\varphi\in H(\ddiv,\Omega)$ satisfies $-\ddiv\varphi=f$,
the $L^2$ orthogonality in \eqref{e:HelmholtzDec} implies that
any $v\in H^1_0(\Omega)$ satisfies
\begin{align*}
 (p,\nabla v)_{L^2(\Omega)} 
  = (\varphi,\nabla v)_{L^2(\Omega)} 
  = (f,v)_{L^2(\Omega)} 
\end{align*}
and, hence, $\widetilde{u}$ solves~\eqref{e:PMPstrongform}.
\end{remark}

\begin{remark}[mixed boundary conditions]\label{r:boundary}
Let $\partial\Omega=\Gamma_D\cup\Gamma_N$
with $\Gamma_D$ closed, $\Gamma_D\cap\Gamma_N=\emptyset$, and 
each connectivity component of $\Gamma_D$ has positive length.
Assume that the triangulation resolves $\Gamma_D$.
Let $H^{-1/2}(\Gamma_N)$ denote the space of generalized normal traces 
of $H(\ddiv,\Omega)$ functions and
let $u_D\in H^1(\Omega)$ and $g\in H^{-1/2}(\Gamma_N)$ 
in the sense that there holds $g=q\cdot\nu$ on $\Gamma_N$ in the 
sense of distributions for some $q\in H(\ddiv,\Omega)$.
Consider the mixed boundary value problem 
$-\Delta u = f$ in $\Omega$ with $u\vert_{\Gamma_D} = u_D$ on $\Gamma_D$
and $(\nabla u\cdot\nu)\vert_{\Gamma_N}=g$ on $\Gamma_N$.
Let $H^1_{D}(\Omega)$ denote the subspace of $H^1(\Omega)$ of functions 
with vanishing trace on $\Gamma_D$. For $\Gamma_D=\emptyset$, define 
$H^1_D(\Omega):=H^1(\Omega)\cap L^2_0(\Omega)$. Define 
\begin{align*}
 H^1_\star(\Omega)&:=\{\beta\in Y\mid
\beta\text{ is constant on each connectivity 
component of }\Gamma_N\}.
\end{align*}
The Helmholtz decomposition
\begin{align*}
 L^2(\Omega;\R^2) 
  =\nabla H^1_{D}(\Omega)\oplus \Curl H^1_\star(\Omega)
\end{align*}
for mixed 
boundary conditions \cite[Corollary~3.1]{GiraultRaviart1986}
then leads to the following formulation.
Let $\varphi\in H(\ddiv,\Omega)$ with $-\ddiv\varphi = f$
additionally fulfil the boundary condition
$\varphi\nu\vert_{\Gamma_N}=g$ and seek 
$(p,\alpha)\in L^2(\Omega;\R^2)\times H^1_\star(\Omega)$
with 
\begin{equation*}
\begin{aligned}
 (p,q)_{L^2(\Omega)} + (q,\Curl \alpha)_{L^2(\Omega)} &= 
(\varphi,q)_{L^2(\Omega)}
 &&\text{ for all }q\in L^2(\Omega;\mathbb{R}^2),\\
 (p,\Curl \beta)_{L^2(\Omega)} &= (\nabla u_D,\Curl\beta)_{L^2(\Omega)} 
 &&\text{ for all }\beta\in H^1_\star(\Omega).
\end{aligned}
\end{equation*}
Since $p=\varphi-\Curl\alpha\in H(\ddiv,\Omega)$,
the equivalence follows as in Subsection~\ref{ss:PMPformulation}
and with 
\begin{align*}
 (p\cdot\nu)\vert_{\Gamma_N}
 = (\varphi\cdot\nu)\vert_{\Gamma_N}
   - (\Curl\alpha\cdot\nu)\vert_{\Gamma_N}
 = g - (\nabla\alpha\cdot\tau)\vert_{\Gamma_N}
 = g.
\end{align*}
\end{remark}

\begin{remark}[multiply connected domains]
If $\Omega\subseteq\R^2$ is a multiply connected polygonal bounded Lipschitz 
domain and $\partial\Omega=\Gamma_D\cup\Gamma_N$, such that all parts of 
$\Gamma_D$ lie on the outer boundary of $\Omega$ (on the unbounded connectivity 
component of $\R^2\setminus\Omega$), then the Helmholtz decomposition of 
Remark~\ref{r:boundary} still holds and
a discretization as above is then immediate.
However, if the Dirichlet boundary $\Gamma_D$ also covers parts of interior 
boundary, that Helmholtz 
decomposition does no longer hold: There exist harmonic functions which are 
constant on different parts of $\Gamma_D$ and, hence, are neither in $\nabla 
H^1_{\Gamma_D}(\Omega)$, nor in $\Curl H^1_\star(\Omega)$.
\end{remark}

\begin{remark}[computation of $\varphi$]\label{r:computationphi}
The computation of $\varphi$ appears as a practical difficulty because
$\varphi$ needs to be defined through an integration of $f$. If $f$ has some simple 
structure, e.g., $f$ is polynomial, this can be done manually, while for more 
complicated $f$, a numerical integration of $f$ has to be employed, 
but is possible in parallel. 
\end{remark}

\subsection{Discretization}\label{ss:PMPdiscretisation}

Let $\tri$ be a regular triangulation of $\Omega$ and 
$k\in\mathbb{N}\cup\{0\}$ and define 
\begin{align*}
 X_h(\tri):=P_k(\tri;\mathbb{R}^2)
 \quad\text{and}\quad
 Y_h(\tri):=P_{k+1}(\tri)\cap Y.
\end{align*}
The discretization of \eqref{e:PMPmixedproblem} seeks 
$p_h\in X_h(\tri)$ and $\alpha_h\in Y_h(\tri)$ with
\begin{subequations}\label{e:PMPdP}
\end{subequations}
%
\begin{align}
 (p_h,q_h)_{L^2(\Omega)} + (q_h,\Curl \alpha_h)_{L^2(\Omega)} &= 
(\varphi,q_h)_{L^2(\Omega)}
 &&\text{ for all }q_h\in X_h(\tri),
  \tag{\ref{e:PMPdP}.a}\label{e:PMPdPeq1}\\
 (p_h,\Curl \beta_h)_{L^2(\Omega)} &= 0 
 &&\text{ for all }\beta_h\in Y_h(\tri).
  \tag{\ref{e:PMPdP}.b}\label{e:PMPdPeq2}
\end{align}

\begin{remark}
Since there are no continuity conditions on $q_h\in X_h(\tri)$
and since $\Curl Y_h(\tri) \subseteq X_h(\tri)$, 
the first equation is fulfilled in a strong form, i.e.,
\begin{align*}
  p_h + \Curl\alpha_h = \Pi_k\varphi.
\end{align*}
In contrast to classical finite element methods, 
the approximation $p_h$ of $\nabla u$ is a gradient only 
in a \emph{discrete orthogonal} sense, namely \eqref{e:PMPdPeq2}. 
For $k=0$, Subsection~\ref{ss:PMPCR} below proves that 
this \emph{discrete orthogonal gradient property} is equivalent to 
being a non-conforming gradient of a Crouzeix-Raviart 
finite element function. 
The main motivation of the novel formulation is the generalization 
of this scheme to any polynomial degree $k$.
\end{remark}

\begin{remark}[existence of discrete solutions]\label{r:dexistence}
Since $\Curl Y_h(\tri) \subseteq X_h(\tri)$, the discrete 
inf-sup condition 
\begin{equation*}
 \left\|\Curl\beta_h\right\|_{L^2(\Omega)}
   \leq \sup_{q_h\in X_h(\tri)\setminus\{0\}} 
       \frac{(q_h,\Curl\beta_h)_{L^2(\Omega)}}{\|q_h\|_{L^2(\Omega)}}
 \qquad\text{for all }\beta_h\in Y_h(\tri)
\end{equation*}
is fulfilled.
This and Brezzi's splitting lemma \cite{Brezzi1974} imply the unique 
existence of a solution $(p_h,\alpha_h)\in X_h(\tri)\times Y_h(\tri)$ 
to \eqref{e:PMPdP}.
The equality in 
\begin{align*}
 \|p_h\|_{L^2(\Omega)}^2 + \left\|\Curl\alpha_h\right\|_{L^2(\Omega)}^2 
   = \|\Pi_k\varphi\|_{L^2(\Omega)}^2
   \leq \|\varphi\|_{L^2(\Omega)}^2.
\end{align*}
follows from the 
$L^2$ orthogonality of $p_h$ and $\Curl \alpha_h$.
\end{remark}

The conformity of the method and the inf-sup conditions from 
Remarks~\ref{r:existence} and \ref{r:dexistence}
imply the following best-approximation result.

\begin{theorem}[best-approximation]\label{t:PMPbestapprox}
The solution $(p,\alpha)\in X\times Y$
to \eqref{e:PMPmixedproblem} and the discrete solution 
$(p_h,\alpha_h)\in X_h(\tri)\times Y_h(\tri)$ of \eqref{e:PMPdP}
satisfy
\begin{align}
  &\|p-p_h\|_{L^2(\Omega)} + \left\|\Curl 
 (\alpha-\alpha_h)\right\|_{L^2(\Omega)}
 \label{e:bestapprox}\\
 &\qquad\qquad 
  \lesssim \Big(\min_{q_h\in X_h(\tri)} \|p-q_h\|_{L^2(\Omega)}
   +\min_{\beta_h\in Y_h(\tri)} 
     \left\|\Curl (\alpha-\beta_h)\right\|_{L^2(\Omega)}\Big).
  \tag*{\qed}
\end{align}
\end{theorem}

\begin{remark}
A direct analysis of the bilinear form 
$\mathcal{B}:\big(X\times Y\big)
\times \big(X\times Y\big) \to\R$ defined by
\begin{align}\label{e:PMPbilinearformB}
 \mathcal{B}((p,\alpha),(q,\beta))
 := (p,q)_{L^2(\Omega)} + (q,\Curl \alpha)_{L^2(\Omega)} + (p,\Curl 
\beta)_{L^2(\Omega)}
\end{align}
for all $p,q\in X$ and all $\alpha,\beta\in Y$
reveals that the inf-sup constant of $\mathcal{B}$ equals $5$
and, hence, the constant hidden in $\lesssim$ in \eqref{e:bestapprox} is $5$.
\end{remark}

\begin{remark}\label{r:globalinfsup}
The best-approximation of Theorem~\ref{t:PMPbestapprox}
contains the term
\begin{align*}
 \min_{\beta_h\in Y_h(\tri)} 
     \left\|\Curl (\alpha-\beta_h)\right\|_{L^2(\Omega)}
\end{align*}
on the right-hand side, which depends on the choice of $\varphi$.
This seems to be worse than the best-approximation 
results for standard FEMs, which do not involve 
such a term. However, if $\varphi$ is chosen smooth enough,
then $\Curl\alpha = \varphi - \nabla u$ has at least 
the same regularity as $\nabla u$, and therefore the convergence 
rate is not diminished.
On the other hand, the approximation space for $p$ does 
not have any continuity restriction and so the first 
approximation term
\begin{align}\label{e:PMPdiscontBestApprox}
  \min_{q_h\in X_h(\tri)} \|p-q_h\|_{L^2(\Omega)}
\end{align}
is superior to the best-approximation of a standard FEM,
where $p=\nabla u$ is approximated with gradients 
of finite element functions. However, 
\cite[Theorem~3.2]{Veeser2014} and the comparison results of 
\cite{CarstensenPeterseimSchedensack2012} prove equivalence 
of \eqref{e:PMPdiscontBestApprox} and the best-approximation with gradients 
of a standard FEM up to some multiplicative constant.
\end{remark}

The following lemma proves a projection property.
This means that for any $v\in H^1_0(\Omega)$, 
the best-approximation of $\nabla v$ in $X_h(\tri)$
is a \emph{discrete orthogonal gradient} in the sense that it is orthogonal 
to $\Curl Y_h(\tri)$ and so belongs to the set of 
\emph{discrete orthogonal gradients} $W_h(\tri)$ defined by
\begin{align}\label{e:PMPMh}
 W_h(\tri) :=\{q_h\in X_h(\tri)\mid (q_h,\Curl\beta_h)_{L^2(\Omega)}=0
\text{ for all }\beta_h\in Y_h(\tri)\}.
\end{align}
The projection property is the key ingredient in the  
optimality analysis of Section~\ref{s:PMPafem}.

\begin{lemma}[projection property]\label{l:PMPintegralmean}
It holds that $W_h(\tri)=\Pi_{X_h(\tri)}\nabla H^1_0(\Omega)$.
Moreover, if $\tri_\star$ is an admissible refinement of $\tri$, 
then $\Pi_{X_h(\tri)}W_h(\tri_\star)= W_h(\tri) $.
\end{lemma}

\begin{proof}
Let $q\in \nabla H^1_0(\Omega)$.
Since $\Curl Y_h(\tri)\subseteq X_h(\tri)$
and $Y_h(\tri)\subseteq Y$, the orthogonality in the Helmholtz 
decomposition~\eqref{e:HelmholtzDec} implies for any $\beta_h\in Y_h(\tri)$ 
that
\begin{align*}
 (\Pi_{X_h(\tri)} q,\Curl\beta_h)_{L^2(\Omega)}
 = (q,\Curl\beta_h)_{L^2(\Omega)}
 =0.
\end{align*}
This proves $\Pi_{X_h(\tri)}\nabla H^1_0(\Omega)\subseteq W_h(\tri)$.
For the converse direction, let $p_h\in W_h(\tri)$ and let $u\in H^1_0(\Omega)$
be a solution (possibly not unique) to 
\begin{align*}
  (\Pi_{X_h(\tri)}\nabla u,\Pi_{X_h(\tri)}\nabla v)_{L^2(\Omega)}
    = (p_h,\Pi_{X_h(\tri)}\nabla v)_{L^2(\Omega)}
    \qquad\text{for all }v\in H^1_0(\Omega).
\end{align*}
The orthogonality of $p_h-\Pi_{X_h(\tri)}\nabla u$ to $\nabla H^1_0(\Omega)$
implies the existence of $\alpha\in Y$ such that 
$p_h-\Pi_{X_h(\tri)}\nabla u=\Curl\alpha$. Therefore, $\Curl\alpha\in X_h(\tri)$
and, hence, $\alpha$ is a piecewise polynomial of degree $\leq k+1$ and 
therefore $\alpha\in Y_h(\tri)$. 
But since $p_h\in W_h(\tri)$, it holds that 
\begin{align*}
  \left\|\Curl \alpha\right\|_{L^2(\Omega)}^2 
    = (p_h-\Pi_{X_h(\tri)}\nabla u,\Curl \alpha)_{L^2(\Omega)}
    =0
\end{align*}
and, hence, $\alpha=0$. This proves $\Pi_{X_h(\tri)}\nabla u=p_h$ and, 
therefore, $W_h(\tri)\subseteq \Pi_{X_h(\tri)}\nabla H^1_0(\Omega)$.

A similar proof applies in the discrete case
and proves $\Pi_{X_h(\tri)} W_h(\tri_\star= W_h(\tri) $.
\end{proof}

\begin{remark}[computational costs]
Problem~\eqref{e:PMPdP} is equivalent to the problem: 
Find $(p_h,\alpha_h)\in X_h(\tri)\times Y_h(\tri)$ such that 
\begin{align*}
  (\Curl\beta_h,\Curl\alpha_h)_{L^2(\Omega)}
    &= (\varphi,\Curl\beta_h)_{L^2(\Omega)}
    \qquad\text{for all }\beta_h\in Y_h(\tri),\\
  p_h&=\Pi_{X_h(\tri)}\varphi - \Curl\alpha_h.
\end{align*}
Therefore, the system matrix is (in 2D) the same than 
that of a standard FEM (up to degrees of freedom on the boundary).
\end{remark}

\subsection{Equivalence with Crouzeix-Raviart FEM}\label{ss:PMPCR}

The non-conforming Crouzeix-Raviart finite element space 
\cite{CrouzeixRaviart1973} reads
\begin{align*}
 \mathrm{CR}^1_0(\tri):=
   \left\{ v_\mathrm{CR}\in P_1(\tri)
    \left\vert 
      \begin{array}{l}
         v_\mathrm{CR}\text{ is continuous at midpoints of interior edges}\\ 
         \text{and vanishes at midpoints of boundary edges}
      \end{array}
  \right\}
  \right..
\end{align*}
Since $\mathrm{CR}^1_0(\tri)\not\subseteq H^1_0(\Omega)$ (if the triangulation 
consists of more than one triangle), the weak gradient of 
a function $v_\mathrm{CR}\in \mathrm{CR}^1_0(\tri)$ does not exist in general. 
However, 
the piecewise version $\nabla_\NC v_\mathrm{CR}\in P_0(\tri;\R^2)$ defined by 
$(\nabla_\NC v_\mathrm{CR})\vert_T:=\nabla(v_\mathrm{CR}\vert_T)$ for all 
$T\in\tri$ exists.
The $P_1$ non-conforming discretization of the Poisson problem 
seeks $u_\mathrm{CR}\in \mathrm{CR}^1_0(\tri)$ with 
\begin{align}\label{e:PMPCRdP}
 (\nabla_\NC u_\mathrm{CR},\nabla_\NC v_\mathrm{CR})_{L^2(\Omega)}
  = (f,v_\mathrm{CR})_{L^2(\Omega)}
  \qquad\text{for all }v_\mathrm{CR}\in\mathrm{CR}^1_0(\tri).
\end{align}
The lowest-order space of Raviart-Thomas 
finite element functions \cite{RaviartThomas1977} 
reads
\begin{align}\label{e:PMPdefRT}
\mathrm{RT}_0(\tri):=\left\{q_\mathrm{RT}\in H(\ddiv,\Omega)\left\vert 
   \begin{array}{l}
     \forall T\in\tri\, \exists a_T\in\R^2, b_T\in \R\\
    \text{with } q_\mathrm{RT}(x) = a_T + b_T x
   \end{array}
   \right\}\right..
\end{align}
The Raviart-Thomas functions have the property that
the integration by parts formula holds for functions in $H^1_0(\Omega)$ 
as well as for functions in $\mathrm{CR}^1_0(\tri)$.

The following proposition proves the equivalence of the 
$P_1$ non-conforming discretization and 
the discretization~\eqref{e:PMPdP} for $k=0$.
Note that the discretization~\eqref{e:PMPCRdP} is a non-conforming discretization,
while the discretization~\eqref{e:PMPdP} is a conforming one.

\begin{proposition}[equivalence with CR-NCFEM]\label{p:PMPequivCR}
Let $f\in P_0(\tri)$ be piecewise constant
and let $\varphi_\mathrm{RT}\in\mathrm{RT}_0(\tri)$ satisfy
$-\ddiv \varphi_\mathrm{RT}= f$.
Then the discrete solution $(p_h,\alpha_h)\in 
P_0(\tri;\R^2)\times (P_1(\tri)\cap Y)$
to \eqref{e:PMPdP}  for $k=0$ and the gradient of the discrete solution 
$u_\mathrm{CR}\in\mathrm{CR}^1_0(\tri)$ to \eqref{e:PMPCRdP} coincide,
\begin{align}\label{e:eqCR}
 p_h = \nabla_\NC u_\mathrm{CR}.
\end{align}
\end{proposition}
 
\begin{proof}
The crucial point is the
discrete Helmholtz decomposition \cite{ArnoldFalk1989}
\begin{align}\label{e:PMPdHelmholtz}
 P_0(\tri;\mathbb{R}^2) = \nabla_\NC \mathrm{CR}^1_0(\tri)
    \oplus \Curl (P_1(\tri)\cap Y).
\end{align}
Since $p_h$ is $L^2$ orthogonal to $\Curl (P_1(\tri)\cap Y)$,
this implies $p_h=\nabla_\NC \widetilde{u}_\mathrm{CR}$
for some $\widetilde{u}_\mathrm{CR}\in\mathrm{CR}^1_0(\tri)$. 
Let $q_h = \nabla_\NC v_\mathrm{CR}$ for some 
$v_\mathrm{CR}\in\mathrm{CR}^1_0(\tri)$. 
Then $q_h$ is 
$L^2$ orthogonal to $\Curl (P_1(\tri)\cap Y)$ and a piecewise integration by 
parts and \eqref{e:PMPdP} imply
\begin{align*}
  (\nabla_\NC \widetilde{u}_\mathrm{CR},\nabla_\NC v_\mathrm{CR})_{L^2(\Omega)}
  &= (p_h,q_h)_{L^2(\Omega)}
 = (\varphi_\mathrm{RT},q_h)_{L^2(\Omega)}\\
 &= (-\ddiv \varphi_\mathrm{RT}, v_\mathrm{CR})_{L^2(\Omega)}
 = (f,v_\mathrm{CR})_{L^2(\Omega)}.
\end{align*}
Hence, $\widetilde{u}_\mathrm{CR}=u_\mathrm{CR} $ 
solves \eqref{e:PMPCRdP}.
\end{proof}

The projection property from Lemma~\ref{l:PMPintegralmean} generalizes 
the famous integral mean property 
$\nabla_\NC I_\NC v = \Pi_{P_0(\tri;\R^2)}\nabla v$ for all 
$v\in H^1_0(\Omega)$ of the non-conforming interpolation operator $I_\NC$.

\begin{remark}[higher polynomial degrees]\label{r:PMPnoteqCR}
For higher polynomial degrees $k\geq 1$,
the discretization~\eqref{e:PMPdP} is not equivalent 
to known non-conforming schemes \cite{FortinSoulie1983,CrouzeixFalk1989,
CrouzeixRaviart1973,MatthiesTobiska2005}, in the sense that 
$W_h(\tri) \neq \nabla_\NC V_h(\tri)$ for those non-conforming
finite element spaces $V_h(\tri)$. 
This follows from 
$\nabla_\NC V_h(\tri)\not\subseteq W_h(\tri)$ for non-conforming FEMs 
with enrichment.
A dimension argument shows 
\begin{align*}
 \mathrm{dim}(W_h(\tri))>\mathrm{dim} V_h(\tri)
\end{align*}
for the non-conforming FEMs of~\cite{FortinSoulie1983,CrouzeixFalk1989} 
without enrichment and therefore 
$W_h(\tri)\neq \nabla_\NC V_h(\tri)$. Moreover, this proves that the
generalization of the projection property to higher moments from 
Lemma~\ref{l:PMPintegralmean} cannot hold for those finite element spaces, 
in contrast to the discretization~\eqref{e:PMPdP}.
\end{remark}

\section{Extensions}\label{s:PMPremarks}

Subsection~\ref{ss:PMPrectangles} generalizes the novel FEM to quadrilateral 
meshes and proves a new discrete Helmholtz decomposition for the $Q_1$ rotated 
non-conforming Rannacher-Turek FEM \cite{RannacherTurek1990}. 
Subsection~\ref{ss:PMPraviartthomas} discusses a discretization with 
Raviart-Thomas functions.

\subsection{Quadrilateral finite elements}\label{ss:PMPrectangles}

For this subsection, consider a regular partition $\tri$ of $\Omega$ in 
quadrilaterals.
Define for the reference rectangle
$\widehat{T}=[0,1]^2$
\begin{align*}
 Q_k(\widehat{T})&:=\{v_h\in P_{2k}(\widehat{T})\mid 
    \exists f,g\in P_k([0,1]):\; v_h(x,y) = f(x)g(y)\}.
\end{align*}
Given $T\in \tri$, let $\psi_T:\widehat{T}\to T$ denote the 
bilinear transformation from the reference rectangle to $T$.
For consistency, let $P_{-1}([0,1]):=\{0\}$ and set   
\begin{align*}
  V_{Q,k}(\tri)&:=\left\{\beta_h\in Y \left\vert \;
     \forall T\in\tri:\; (\beta_h\circ \psi_T)\vert_{\widehat{T}}\in 
Q_k(\widehat{T})\right\}\right.,\\
  X_k^{\mathrm{rect}}(\widehat{T})
   &:=\left\{\tau_h\in L^2(\widehat{T};\R^2) \left|
     \begin{array}{l}
       \exists a\in \R, b,c\in P_{k-2}([0,1]), d,e\in 
        Q_{k-1}(\widehat{T})\\
        \text{ such that } \forall (\widehat{x},\widehat{y})\in \widehat{T}\\
        \tau_h(\widehat{x},\widehat{y}) = a\begin{pmatrix}
                         -\widehat{x}^k \widehat{y}^{k-1}\\ 
                          \widehat{x}^{k-1} \widehat{y}^k
                       \end{pmatrix}
          + \begin{pmatrix}
              \widehat{x}^k b(\widehat{y}) + d(\widehat{x},\widehat{y})\\ 
              \widehat{y}^k c(\widehat{x})+e(\widehat{x},\widehat{y})
            \end{pmatrix}
     \end{array}
    \right\}\right.,\\
  X_k^{\mathrm{rect}}(\tri)&:=\left\{\tau_h\in L^2(\Omega;\R^2) \left\vert 
        \begin{array}{l}
           \forall T\in\tri\, \exists \rho_T\in X_k^{\mathrm{rect}}(\widehat{T})
              \text{ such that }\\
           (\tau_h\circ \psi_T)\vert_{\widehat{T}}\\
              \qquad    =  \begin{pmatrix}
                             0 & 1 \\ -1 & 0
                           \end{pmatrix}
                    D(\psi_T^{-1})^\top \circ\psi_T
                        \begin{pmatrix}
                          0 & -1 \\ 1 & 0
                        \end{pmatrix}
                   \rho_T
        \end{array}
    \right\}\right..
\end{align*}
Then a discretization with respect to the quadrilateral partition 
seeks $p_h\in X_k^{\mathrm{rect}}(\tri)$ and 
$\alpha_h\in V_{Q,k}(\tri)$
with
\begin{align*}
 (p_h,q_h)_{L^2(\Omega)} + (q_h,\Curl \alpha_h)_{L^2(\Omega)} &= 
(\varphi,q_h)_{L^2(\Omega)}
 &&\text{ for all }q_h\in X_k^{\mathrm{rect}}(\tri),\\
 (p_h,\Curl \beta_h)_{L^2(\Omega)} &= 0 
 &&\text{ for all }\beta_h\in V_{Q,k}(\tri).
\end{align*}

Let $\beta_h\in V_{Q,k}(\tri)$, i.e., $(\beta_h\circ 
\psi_T)\vert_{\widehat{T}}\in Q_k(\widehat{T})$. A direct calculation 
reveals for all 
$T\in\tri$
\begin{align*}
 ((\Curl \beta_h)^\top\circ\psi_T)\vert_T
  &= 
   (\nabla (\beta_h\circ\psi_T\circ\psi_T^{-1}))^\top\circ\psi_T
     \begin{pmatrix}
       0 & 1\\ -1 & 0
     \end{pmatrix}\\
  &= (\nabla(\beta_h\circ\psi_T))^\top
     D(\psi_T^{-1}) \circ\psi_T
     \begin{pmatrix}
       0 & 1\\ -1 & 0
     \end{pmatrix}.
\end{align*}
Let $(\beta_h\circ \psi_T)(\widehat{x},\widehat{y})
=(\mathfrak{a} \widehat{x}^k+f(\widehat{x}))(\mathfrak{b} \widehat{y}^k+g(\widehat{y}))$ 
with $\mathfrak{a},\mathfrak{b}\in\R$ and $f,g\in P_{k-1}([0,1])$. Then it holds
\begin{align*}
 \nabla(\beta_h\circ\psi_T) = 
    \mathfrak{a}\mathfrak{b}k \begin{pmatrix}
                                \widehat{x}^{k-1}\widehat{y}^k\\
                                \widehat{y}^{k-1} \widehat{x}^k
                              \end{pmatrix}
   + \begin{pmatrix}
       \mathfrak{b} \widehat{y}^k \partial f(\widehat{x})/\partial \widehat{x}\\
       \mathfrak{a} \widehat{x}^k \partial g(\widehat{y})/\partial \widehat{y}
     \end{pmatrix}
  + \begin{pmatrix}
      \mathfrak{a} k \widehat{x}^{k-1} g(\widehat{y}) 
            + g(\widehat{y}) \partial f(\widehat{x})/\partial \widehat{x}\\
      \mathfrak{b} k \widehat{y}^{k-1} f(\widehat{x}) 
            + f(\widehat{x}) \partial g(\widehat{y})/\partial \widehat{y}
    \end{pmatrix}
\end{align*}
and therefore 
\begin{align*}
 &(\nabla(\beta_h\circ\psi_T))^\top 
   \begin{pmatrix}
       0 & -1\\ 1 & 0
   \end{pmatrix}\\
  &\quad\quad = 
    \left(\mathfrak{a}\mathfrak{b}k \begin{pmatrix}
                                \widehat{y}^{k-1} \widehat{x}^k\\
                                -\widehat{x}^{k-1}\widehat{y}^k
                              \end{pmatrix}
   + \begin{pmatrix}
       \mathfrak{a} \widehat{x}^k \partial g(\widehat{y})/\partial \widehat{y}\\
       -\mathfrak{b} \widehat{y}^k \partial f(\widehat{x})/\partial \widehat{x}
     \end{pmatrix}
  + \begin{pmatrix}
      \mathfrak{b} k \widehat{y}^{k-1} f(\widehat{x}) 
                   + f(\widehat{x}) \partial g(\widehat{y})/\partial \widehat{y}\\
      -\mathfrak{a} k \widehat{x}^{k-1} g(\widehat{y}) 
                   - g(\widehat{y}) \partial f(\widehat{x})/\partial \widehat{x}
    \end{pmatrix}
     \right)^\top\\
 & \quad\quad =:(\rho_T(\widehat{x},\widehat{y}))^\top.
\end{align*}
This implies $\rho_T\in X_k^{\mathrm{rect}}(\widehat{T})$ and 
$(\nabla(\beta_h\circ\psi_T))= (0,1;-1,0)\rho_T$. 
The combination of the previous equalities leads to 
\begin{align*}
  ((\Curl \beta_h)\circ\psi_T)\vert_T
   =  \begin{pmatrix}
                             0 & 1 \\ -1 & 0
                           \end{pmatrix}
                    D(\psi_T^{-1})^\top \circ\psi_T
                        \begin{pmatrix}
                          0 & -1 \\ 1 & 0
                        \end{pmatrix}
   \rho_T.
\end{align*}
Consequently, $\Curl \beta_h\in 
X_k^{\mathrm{rect}}(\tri)$.
This and the conformity of the method 
prove as in Section~\ref{s:PMPformulation} the following statements
\begin{itemize}
 \item[(i)] unique existence of solutions,
 \item[(ii)] the best-approximation 
    result
    \begin{align*}
      \|p-p_h&\|_{L^2(\Omega)} + \left\|\Curl 
(\alpha-\alpha_h)\right\|_{L^2(\Omega)}\\
    &\lesssim \Big(\min_{q_h\in X_k^{\mathrm{rect}}(\tri)} 
\|p-q_h\|_{L^2(\Omega)}
      +\min_{\beta_h\in V_{Q,k}(\tri)} 
	\left\|\Curl (\alpha-\beta_h)\right\|_{L^2(\Omega)}\Big),
    \end{align*}
 \item[(iii)] the projection property
      \begin{align*}
	\Pi_{X_k^{\mathrm{rect}}(\tri)} \nabla H^1_0(\Omega)
	\subseteq W_h^{\mathrm{rect}}(\tri)
      \end{align*}
      for 
      \begin{align}\label{e:PMPrectdefMh}
	W_h^{\mathrm{rect}}(\tri) 
	  = \{q_h\in X_k^{\mathrm{rect}}(\tri)\mid 
            \forall \beta_h\in V_{Q,k}(\tri):\;
                   (q_h,\Curl\beta_h)_{L^2(\Omega)}=0\}.
      \end{align}
\end{itemize}

\begin{remark}
The properties (i)--(iii) still hold for any $\widetilde{X}_h(\tri)$ with 
$X_k^\mathrm{rect}(\tri)\subseteq\widetilde{X}_h(\tri)\subseteq X$.
\end{remark}

The remaining part of this subsection proves the equivalence 
of the lowest-order rectangular discretization with 
the non-conforming Rannacher-Turek FEM \cite{RannacherTurek1990}.
To this end, define for the reference rectangle $\widehat{T}$
and the bilinear transformation $\psi_T:\widehat{T}\to T$,
\begin{align}
  Q^{\mathrm{rot}}(\widehat{T})&:=\mathrm{span}\{1,x,y,x^2-y^2\},
  \notag\\
   V^{\mathrm{rot}}_\NC(\tri)&
    :=\left\{v_h\in L^2(\Omega) \left\vert
        \begin{array}{l}
            \forall T\in \tri:\; (v_h\circ\psi_T)\vert_{\widehat{T}}
          \in Q^{\mathrm{rot}}(\widehat{T}) 
          \text{ and }\\\int_E v_h\,ds
          \text{ is continuous}
          \text{ for all interior}\\
          \text{edges }E
          \text{ and vanishes at boundary edges }E
        \end{array}
    \right\}\right..
    \label{e:PMPQ1rotdef}
\end{align}

The following lemma proves a relation between the cardinalities of the 
quadrilaterals, nodes, and interior edges of a quadrilateral partition similar to 
Euler's formulae
\begin{equation}\label{e:PMPeuler}
\begin{aligned}
 \mathrm{card}(\edges) + \mathrm{card}(\edges(\Omega)) &= 3\,\mathrm{card}(\tri),\\
  \mathrm{card}(\edges(\Omega)) + \mathrm{card}(\nodes) 
   &= 2\,\mathrm{card}(\tri) +1
\end{aligned}
\end{equation}
on triangles.
This enables a dimension 
argument in the proof of the discrete Helmholtz decomposition in 
Theorem~\ref{t:PMPrectdHD} below.

\begin{lemma}[Euler formula for quadrilateral partitions]\label{l:PMPeulerrect}
Let $\tri$ be a regular partition of $\Omega$ in quadrilaterals
with edges $\edges$, interior edges $\edges(\Omega)$, and vertices $\nodes$.
Then it holds that $3\mathrm{card}(\tri) + 1 
    = \mathrm{card}(\edges(\Omega)) + \mathrm{card}(\nodes)$.
\end{lemma}

\begin{proof}
Define a triangulation $\tri_\Delta$ of $\Omega$ in triangles 
by the division of each quadrilateral into two triangles by a diagonal cut.
Let $\edges_\Delta$ denote the edges of $\tri_\Delta$, 
$\edges_\Delta(\Omega)$ the interior edges and $\nodes_\Delta$ 
the vertices.
Then the following relations between the two partitions hold
\begin{align*}
 \mathrm{card}(\tri_\Delta) &= 2\mathrm{card}(\tri),
 \qquad 
 &\mathrm{card}(\edges_\Delta) &= \mathrm{card}(\edges) + \mathrm{card}(\tri),\\
 \mathrm{card}(\edges_\Delta(\Omega)) &= \mathrm{card}(\edges(\Omega)) + 
\mathrm{card}(\tri),
 \qquad 
 &\mathrm{card}(\nodes_\Delta)&=\mathrm{card}(\nodes).
\end{align*}
This and Euler's formulae for triangles \eqref{e:PMPeuler}
prove  
\begin{align*}
 \mathrm{card}(\edges(\Omega))+\mathrm{card}(\nodes) 
 &= \mathrm{card}(\edges_\Delta(\Omega)) -\mathrm{card}(\tri) + 
\mathrm{card}(\nodes_\Delta)\\
 &= 2\mathrm{card}(\tri_\Delta) + 1 - \mathrm{card}(\tri) 
 = 3\mathrm{card}(\tri) + 1.
 \qedhere
\end{align*}
\end{proof}

The following theorem proves that the solution space $W_h^\mathrm{rect}(\tri)$ 
from~\eqref{e:PMPrectdefMh} equals the piecewise gradients of functions in 
$V^{\mathrm{rot}}_\NC(\tri)$ on a partition in squares for $k=1$.

\begin{theorem}[discrete Helmholtz decomposition on squares] 
\label{t:PMPrectdHD}
Let $\tri$ be a regular partition of $\Omega$ in squares.
Then,
\begin{align}\label{e:dHelmholtzdecRect}
  X_1^{\mathrm{rect}}(\tri) = \nabla_\NC V^{\mathrm{rot}}_\NC(\tri)
      \oplus \Curl V_{Q,1}(\tri)
\end{align}
and the decomposition is $L^2$ orthogonal.
\end{theorem}

\begin{remark}
The $L^2$-orthogonality in \eqref{e:dHelmholtzdecRect}
still holds for a partition in parallelograms. However, 
$\nabla_\NC V^{\mathrm{rot}}_\NC(\tri)\not\subseteq X_1^{\mathrm{rect}}(\tri)$
for general quadrilateral partitions.
\end{remark}

\begin{proof}[Proof of Theorem~\ref{t:PMPrectdHD}]
Let $v_h\in V^{\mathrm{rot}}_\NC(\tri)$ and 
$\beta_h\in V_{Q,1}(\tri)$. A piecewise integration by parts leads to 
\begin{align*}
 (\nabla_\NC v_h,\Curl\beta_h)_{L^2(\Omega)}
  = \sum_{E\in\edges} \int_E [v_h]_E \nabla\beta_h\cdot\tau_E\,ds.
\end{align*}
Since $\tri$ consists of parallelograms, the bilinear transformation 
$\psi_T:\widehat{T}\to T$ is affine and, hence,
$\beta_h\vert_E$ is affine on each edge $E\in\edges$.
This implies that $\nabla\beta_h\cdot\tau_E$ is constant.
Since the integral mean of $[v_h]_E$ vanishes,
this proves the $L^2$ orthogonality.

Let $v_h\in V^{\mathrm{rot}}_\NC(\tri)$. A computation reveals for all 
$T\in\tri$ that there exist $f_T\in \R$ and $g_T\in \R^2$ 
such that
\begin{align*}
 \nabla v_h (x,y) 
   =  D(\psi_T^{-1})^\top
   \bigg( f_T \begin{pmatrix}
                 -x \\ y
               \end{pmatrix}
        + g_T\bigg).
\end{align*}
For $k=1$, $X_1^{\mathrm{rect}}(\tri)$ reads
\begin{align*}
  X_1^{\mathrm{rect}}(\tri)
  = \left\{\tau_h\in L^2(\Omega;\R^2) \left\vert 
        \begin{array}{l}
           \forall T\in\tri\, \exists a_T\in\R, d_T\in\R^2
              \text{ such that }\\
           (\tau_h\circ \psi_T)\vert_{\widehat{T}}
                  =
                          \begin{pmatrix}
                             0 & 1 \\ -1 & 0
                           \end{pmatrix}
                   ( D(\psi_T^{-1})^\top 
                  \\ 
     \qquad\qquad\qquad
              \circ\psi_T) 
                        \begin{pmatrix}
                          0 & -1 \\ 1 & 0
                        \end{pmatrix}
              \left(a_T \begin{pmatrix}
                          -x \\ y
                        \end{pmatrix}
                        + d_T\right)
        \end{array}
    \right\}\right..
\end{align*}
Since all $T\in\tri$ are squares, $D\psi_T$ and 
$( 0, 1; -1, 0)$ commute, 
and, hence, $\nabla v_h\in X_1^{\mathrm{rect}}(\tri)$.
Thus, $\nabla_\NC V^{\mathrm{rot}}_\NC(\tri)
\oplus \Curl V_{Q,1}(\tri) \subseteq X_1^{\mathrm{rect}}(\tri)$.
The dimension of $\nabla_\NC V^{\mathrm{rot}}_\NC(\tri)$ equals 
$\mathrm{card}(\edges(\Omega))$ and the dimension of $\Curl V_{Q,1}(\tri)$
equals $\mathrm{card}(\nodes)-1$, while the dimension of $X_1^{\mathrm{rect}}(\tri)$
equals $3\mathrm{card}(\tri)$. This and Lemma~\ref{l:PMPeulerrect}
prove the assertion.
\end{proof}

\begin{remark}[arbitrary quadrilaterals]
The best-approximation (ii) from above proves quasi-optimal 
convergence even for arbitrary quadrilaterals.
Standard interpolation error estimates for  $V_{Q,1}(\tri)$ 
and for $P_0(\tri;\R^2)\subseteq X^\mathrm{rect}_1(\tri)$ \cite{Ciarlet1978}
lead to first-order convergence rates of $h$ for sufficiently smooth solutions.
This should be contrasted with 
\cite{RannacherTurek1990}, where quasi-optimal convergence is only obtained for 
a modification of~\eqref{e:PMPQ1rotdef} where $V^{\mathrm{rot}}_\NC(\tri)$ 
is defined in terms of local coordinates.
\end{remark}

\subsection{Relation to mixed Raviart-Thomas FEM}\label{ss:PMPraviartthomas}

This subsection shows that the classical mixed Raviart-Thomas FEM \cite{RaviartThomas1977}
can be regarded as a particular choice of the ansatz spaces in the new mixed scheme.

Let $\tri$ denote a regular triangulation of $\Omega$ in triangles.
Define the space of Raviart-Thomas functions \cite{RaviartThomas1977}
\begin{align*}
 X_\mathrm{RT}(\tri) = \{q_\mathrm{RT}\in H(\ddiv,\Omega)\mid \forall T\in\tri: \; 
     q_\mathrm{RT}\vert_T(x) \in P_{k}(T;\R^2) + P_{k}(T)\, x\}
\end{align*}
and 
\begin{align*}
 Y_\mathrm{RT}(\tri):= P_{k+1}(\tri)\cap Y.
\end{align*}
Then the following problem is a discretization of \eqref{e:PMPmixedproblem}: 
Seek $(p_\mathrm{RT},\alpha_\mathrm{RT})
\in X_\mathrm{RT}(\tri)\times Y_\mathrm{RT}(\tri)$ with 
\begin{equation}\label{e:PMPmixedRT}
\begin{aligned}
 (p_\mathrm{RT},q_\mathrm{RT})_{L^2(\Omega)} 
    + (q_\mathrm{RT},\Curl\alpha_\mathrm{RT})_{L^2(\Omega)}
  &= (\varphi,q_\mathrm{RT})
  &\;\;\text{for all }q_\mathrm{RT}\in X_\mathrm{RT}(\tri),\\
 (p_\mathrm{RT},\Curl\beta_\mathrm{RT})_{L^2(\Omega)} &= 0
  &\;\;\text{for all }\beta_\mathrm{RT}\in Y_\mathrm{RT}(\tri).
\end{aligned}
\end{equation}
Since $\Curl Y_\mathrm{RT}(\tri)\subseteq P_{k}(\tri;\R^2)$ and 
$\ddiv \Curl v_\mathrm{RT}=0$ for all 
$v_\mathrm{RT}\in Y_\mathrm{RT}(\tri)$, it follows 
$\Curl Y_\mathrm{RT}(\tri)\subseteq X_\mathrm{RT}(\tri)$. This and the 
conformity of the method 
guarantee as in Section~\ref{s:PMPformulation} and in 
Subsection~\ref{ss:PMPrectangles} the unique existence of solutions, a 
best-approximation result, and the projection property 
\begin{align*}
 \Pi_{X_\mathrm{RT}(\tri)} \nabla H^1_0(\Omega) 
   &\subseteq W_\mathrm{RT}(\tri) \\
   &:=\{q_\mathrm{RT}\in X_\mathrm{RT}(\tri)\mid 
      \forall \beta_\mathrm{RT}\in Y_\mathrm{RT}(\tri):\; 
   (q_\mathrm{RT},\Curl\beta_\mathrm{RT})_{L^2(\Omega)}=0\}.
\end{align*}
The discrete Helmholtz decomposition of 
\cite{HuangXu2012,ArnoldFalkWinther1997,BrezziFortinStenberg1991} proves 
\begin{align*}
 X_\mathrm{RT}(\tri) = \nabla_{\!\mathrm{RT}} P_{k}(\tri) \oplus \Curl Y_\mathrm{RT}(\tri)
\end{align*}
with the operator $\nabla_{\!\mathrm{RT}}:P_{k}(\tri)\to X_\mathrm{RT}(\tri)$ 
defined for all $v_\mathrm{RT}\in P_{k}(\tri)$ by 
\begin{align*}
 (\nabla_{\!\mathrm{RT}} v_\mathrm{RT},q_\mathrm{RT})_{L^2(\Omega)} 
   = -(v_\mathrm{RT},\ddiv q_\mathrm{RT})_{L^2(\Omega)}
 \qquad\text{for all }q_\mathrm{RT}\in X_\mathrm{RT}(\tri).
\end{align*}
This decomposition yields the equivalence of \eqref{e:PMPmixedRT} with the 
problem: Seek $(p_\mathrm{RT},\widetilde{u}_\mathrm{RT})
\in X_\mathrm{RT}(\tri)\times P_{k}(\tri)$ 
with
\begin{align*}
 p_\mathrm{RT} &= \nabla_{\!\mathrm{RT}} \widetilde{u}_\mathrm{RT},\\
 (w_\mathrm{RT},\ddiv p_\mathrm{RT})_{L^2(\Omega)} & = 
(\ddiv\Pi_{X_{\mathrm{RT}}(\tri)}\varphi,w_\mathrm{RT})_{L^2(\Omega)}
 \qquad\text{for all }w_\mathrm{RT} \in P_{k}(\tri).
\end{align*}
This is the classical Raviart-Thomas discretization with $f$ replaced by
$\ddiv\Pi_{X_{\mathrm{RT}}(\tri)}\varphi$.

Assume now that the right-hand side $\varphi\in X_\mathrm{RT}(\tri)$ is 
a Raviart-Thomas function. Since by definition $Y_\mathrm{RT}(\tri)=Y_h(\tri)$ 
with $Y_h(\tri)$ from Subsection~\ref{ss:PMPdiscretisation} and since 
$\alpha_\mathrm{RT}$ is the solution of 
\begin{align*}
  (\Curl \beta_\mathrm{RT},\Curl \alpha_\mathrm{RT})_{L^2(\Omega)}
    = (\varphi,\Curl \beta_\mathrm{RT})_{L^2(\Omega)}
    \qquad \text{for all }\beta_\mathrm{RT}\in Y_\mathrm{RT}(\tri),
\end{align*}
it holds 
$\alpha_\mathrm{RT}=\alpha_h$ with $\alpha_h$ from~\eqref{e:PMPdP}.
Since $\varphi = p_\mathrm{RT}+\Curl\alpha_\mathrm{RT}$ and 
$\Pi_{X_h(\tri)}\varphi = p_h+\Curl\alpha_h$, it follows 
\begin{align*}
 p_h = \Pi_{X_h(\tri)} p_\mathrm{RT}.
\end{align*}
For $k=0$, the equivalence with the Crouzeix-Raviart FEM~\eqref{e:eqCR}
then proves the identity 
\begin{align*}
  \nabla_\NC u_\mathrm{CR}=\Pi_{X_h(\tri)} p_\mathrm{RT},
\end{align*}
which is also known as Marini identity~\cite{ArnoldBrezzi1985,Marini1985}.

\section{Medius analysis}\label{s:PMPmedius}

The medius analysis of \cite{Gudi2010,CarstensenPeterseimSchedensack2012}
proves for the discrete solution 
$u_\mathrm{CR}\in\mathrm{CR}^1_0(\tri)$ to \eqref{e:PMPCRdP} 
the best-approximation result 
\begin{align}\label{e:PMPmediusclassic}
 \|\nabla_\NC(u-u_\mathrm{CR})\|_{L^2(\Omega)}
  \lesssim \min_{v_\mathrm{CR}\in\mathrm{CR}^1_0(\tri)}
    \|\nabla_\NC(u-v_\mathrm{CR})\|_{L^2(\Omega)}
   + \mathrm{osc}(f,\tri).
\end{align}
The following theorem proves a generalization  
for the discretization \eqref{e:PMPdP} for the lowest order case $k=0$.

\begin{theorem}[best-approximation property]\label{t:PMPmedius}
Let $(p,\alpha)\in X\times Y$ be the solution to \eqref{e:PMPmixedproblem} and 
$(p_h,\alpha_h)\in P_0(\tri;\R^2)\times (P_1(\tri)\cap Y) $
be the solution to \eqref{e:PMPdP}.
Then the following best-approximation result holds
\begin{equation}\label{e:PMPmedius}
\begin{aligned}
 \| p-p_h\|_{L^2(\Omega)} 
  & \lesssim \|p-\Pi_0 p\|_{L^2(\Omega)} + \mathrm{osc}(f,\tri)\\
 &\qquad
  + \sup_{v_\mathrm{CR}\in\mathrm{CR}^1_0(\tri)\setminus\{0\}}
        \frac{(f,v_\mathrm{CR})_{L^2(\Omega)} 
    - (\varphi,\nabla_\NC v_\mathrm{CR})_{L^2(\Omega)}} 
                 {\|\nabla_\NC v_\mathrm{CR}\|_{L^2(\Omega)} }.
\end{aligned}
\end{equation}
\end{theorem}

\begin{remark}
If $\varphi$ is a lowest-order Raviart-Thomas function, then it allows for 
an integration by parts formula also with Crouzeix-Raviart functions 
(see Subsection~\ref{ss:PMPCR}). Therefore, the third term on the right-hand 
side of \eqref{e:PMPmedius} vanishes. This and the equivalence with the 
non-conforming FEM of Crouzeix and Raviart from Subsection~\ref{ss:PMPCR} 
reveal the best-approximation result \eqref{e:PMPmediusclassic}.
\end{remark}

The remaining part of this section is devoted to the proof of 
Theorem~\ref{t:PMPmedius}.
The following lemma from \cite{CarstensenSchedensack2014,
CarstensenGallistlSchedensack2014} is the key ingredient of this proof.
Recall the definition of $\mathrm{CR}^1_0(\tri)$ from Subsection~\ref{ss:PMPCR}.

\begin{lemma}[companion]\label{l:PMPcompanion}
For any $v_\mathrm{CR}\in\mathrm{CR}^1_0(\tri)$ there exists 
$v\in H^1_0(\Omega)$ with the following properties
\begin{align}
 &\text{(i) } &&\Pi_0 \nabla_\NC (v-v_\mathrm{CR})=0,\notag\\
 &\text{(ii) }&&\Pi_0 (v-v_\mathrm{CR}) = 0,\notag\\
 &\text{(iii) }&&\|h_\tri^{-1}(v_\mathrm{CR}-v)\|_{L^2(\Omega)}
        + \|\nabla_\NC(v_\mathrm{CR}-v)\|_{L^2(\Omega)}
             \lesssim \|\nabla_\NC v_\mathrm{CR}\|_{L^2(\Omega)}.
 \tag*{\qed}
\end{align}
\end{lemma}

\begin{proof}[Proof of Theorem~\ref{t:PMPmedius}]
Define $q_h:=\Pi_0 p - p_h\in P_{0}(\tri;\R^2)$.
The projection property of Lemma~\ref{l:PMPintegralmean} implies that 
$q_h\in W_h(\tri)$ and the discrete Helmholtz decomposition \eqref{e:PMPdHelmholtz}
guarantees the existence of $v_\mathrm{CR}\in\mathrm{CR}^1_0(\tri)$ with 
$q_h=\nabla_\NC v_\mathrm{CR}$.
Let $v\in H^1_0(\Omega)$ denote the companion of $v_\mathrm{CR}$ from 
Lemma~\ref{l:PMPcompanion}. 
Then 
\begin{equation}\label{e:PMPproofmedius1}
\begin{aligned}
 (p-p_h,q_h)_{L^2(\Omega)} 
  &= (p,\nabla_\NC(v_\mathrm{CR}-v))_{L^2(\Omega)} 
   + (p,\nabla v)_{L^2(\Omega)}\\
  &\qquad\qquad\qquad - (p_h,\nabla_\NC v_\mathrm{CR})_{L^2(\Omega)}.
\end{aligned}
\end{equation}
The properties (i) and (iii) from Lemma~\ref{l:PMPcompanion} yield for the 
first term on the right-hand side 
\begin{equation}\label{e:PMPproofmedius2}
\begin{aligned}
 (p,\nabla_\NC(v_\mathrm{CR}-v))_{L^2(\Omega)} 
 & = (p-\Pi_0 p,\nabla_\NC(v_\mathrm{CR}-v))_{L^2(\Omega)}\\
 &  \lesssim \|p-\Pi_0 p\|_{L^2(\Omega)} \;\|\nabla_\NC v_\mathrm{CR}\|_{L^2(\Omega)}.
\end{aligned}
\end{equation}
The problems \eqref{e:PMPmixedproblem} and \eqref{e:PMPdP} lead
for the second and third term on the right-hand side of 
\eqref{e:PMPproofmedius1}
to 
\begin{align*}
 (p,\nabla v)_{L^2(\Omega)} - (p_h,\nabla_\NC v_\mathrm{CR})_{L^2(\Omega)} 
   = (\varphi,\nabla v)_{L^2(\Omega)} 
  -(\varphi,\nabla_\NC v_\mathrm{CR})_{L^2(\Omega)}.
\end{align*}
Since $-\ddiv\varphi=f$, it follows
\begin{align*}
& (\varphi,\nabla v)_{L^2(\Omega)} 
  -(\varphi,\nabla_\NC v_\mathrm{CR})_{L^2(\Omega)}\\
 &\qquad\qquad\qquad
  = (f,v-v_\mathrm{CR})_{L^2(\Omega)}
   + (f,v_\mathrm{CR})_{L^2(\Omega)} - (\varphi,\nabla_\NC v_\mathrm{CR})_{L^2(\Omega)}.
\end{align*}
Properties (ii) and (iii) of Lemma~\ref{l:PMPcompanion} prove
\begin{align*}
  = (f,v-v_\mathrm{CR})_{L^2(\Omega)}
  \lesssim \mathrm{osc}(f,\tri) \|\nabla_\NC v_\mathrm{CR}\|_{L^2(\Omega)}.
\end{align*}
The combination with \eqref{e:PMPproofmedius1} and \eqref{e:PMPproofmedius2}
and a Cauchy inequality yield 
\begin{align*}
 (p-p_h,q_h)_{L^2(\Omega)} 
  & \lesssim \Bigg(\|p-\Pi_0 p\|_{L^2(\Omega)}
   + \mathrm{osc}(f,\tri) \\
 &\quad
  + \sup_{v_\mathrm{CR}\in\mathrm{CR}^1_0(\tri)\setminus\{0\}}
        \frac{(f,v_\mathrm{CR})_{L^2(\Omega)} 
    - (\varphi,\nabla_\NC v_\mathrm{CR})_{L^2(\Omega)}} 
                 {\|\nabla_\NC v_\mathrm{CR}\|_{L^2(\Omega)} }\Bigg)
     \|q_h\|_{L^2(\Omega)}.
\end{align*}
This and 
\begin{align*}
 \|p-p_h\|_{L^2(\Omega)}^2 
   = \|p-\Pi_0 p\|_{L^2(\Omega)}^2 + \|q_h\|_{L^2(\Omega)}^2
   = \|p-\Pi_0 p\|_{L^2(\Omega)}^2 
     + (p-p_h,q_h)_{L^2(\Omega)} 
\end{align*}
prove the assertion.
\end{proof}

\begin{remark}[higher polynomial degrees] \label{r:PMPcompanionhigherpolynom}
For $k\geq 1$, Remark~\ref{r:PMPnoteqCR} implies that 
an analogue of Lemma~\ref{l:PMPcompanion} cannot 
be proved in the same way.
\end{remark}

\section{Adaptive algorithm}\label{s:PMPafem}

This section defines an adaptive algorithm based on separate marking 
and proves its quasi-optimal convergence.

\subsection{Adaptive algorithm and optimal convergence rates} 
\label{ss:PMPafemdef}

Let $\tri_0$ denote some initial shape-regular triangulation of $\Omega$, such 
that each triangle $T\in\tri$ is equipped with a refinement edge 
$E_T\in\edges(T)$. A proper choice of these refinement edges guarantees an 
overhead control \cite{BinevDahmenDeVore2004}.

Let $\mathbb{T}(N)$ denote the subset of $\mathbb{T}$ of all admissible 
triangulations with at most $\mathrm{card}(\tri_0) + N$ triangles.
The adaptive algorithm involves the overlay of two admissible triangulations
$\tri,\tri_\star\in\mathbb{T}$, which reads
\begin{align}\label{e:defoverlay}
  \tri\otimes\tri_\star:=\{T\in\tri\cup\tri_\star\mid 
      \exists K\in\tri,K_\star\in \tri_\star\text{ with }
       T\subseteq K\cap K_\star\}.
\end{align}

Given a triangulation $\tri_\ell$,
define for all $T\in\tri_\ell$ the local error estimator contributions by
\begin{equation}\label{e:PMPdefmulambda1}
\begin{aligned}
  \lambda^2(\tri_\ell,T)&:= \|h_\tri \curl_\NC p_h\|_{L^2(T)}^2
      + h_T \sum_{E\in \mathcal{E}(T)}  \|[p_h]_E\cdot \tau_E \|_{L^2(E)}^2,\\
  \mu^2(T)&:= \|\varphi-\Pi_k\varphi\|_{L^2(T)}^2
\end{aligned}
\end{equation}
and the global error estimators by
\begin{equation}\label{e:PMPdefmulambda2}
\begin{aligned}
 \lambda_\ell^2
&:=\lambda^2(\tri_\ell,\tri_\ell)
 &&\text{with}&
 \lambda^2(\tri_\ell,\mathcal{M})&:=&&\sum_{T\in \mathcal{M}} 
\lambda^2(\tri_\ell,T)
 &&
\text{for any }\mathcal{M}\subseteq\tri_\ell,
\\
 \mu^2_\ell
&:=\mu^2(\tri_\ell)
 &&\text{with}&
 \mu^2(\mathcal{M})&:=&&\sum_{T\in\mathcal{M}} \mu^2(T)
 && 
\text{for any }\mathcal{M}\subseteq\tri_\ell.
\end{aligned}
\end{equation}
The adaptive algorithm is driven by these two error estimators
and runs the following loop.

\begin{algo}[AFEM]
  \label{a:PMPafem}
\begin{algorithmic}
\State 
\Require Initial triangulation $\tri_0$, parameters $0<\theta_A\le 1$,
    $0<\rho_B<1$, $0<\kappa$.
\For{$\ell=0,1,2,\dots$}
 \State {\it Solve.}
 Compute solution $(p_{\ell},\alpha_{\ell})\in X_h(\tri_\ell)\times 
Y_h(\tri_\ell)$ 
 of \eqref{e:PMPdP} with respect
 \State  \qquad
 to $\tri_\ell$.
 \State {\it Estimate.}
 Compute local contributions of the error estimators
 $\big(\lambda^2(\tri_\ell,T)\big)_{T\in\tri_\ell}$ 
 \State \qquad
 and 
 $(\mu^2(T))_{T\in\tri_\ell}$.
 \If{$\mu_\ell^2\leq \kappa\lambda_\ell^2$}
 \State   
 {\it Mark.}
 The D\"orfler marking chooses a minimal subset 
 $\mathcal{M}_\ell\subseteq\tri_\ell$
 such that
 \State \qquad
 $
    \theta_A \lambda_\ell^2 
 \le  \lambda_\ell^2 (\tri_\ell,\mathcal{M}_\ell)
 $.
 \State
 {\it Refine.}
 Generate the smallest admissible refinement
 $\tri_{\ell+1}$ of $\tri_\ell$ in which 
 \State \qquad at least all 
 triangles in $\mathcal{M}_\ell$ are refined.
 \Else 
 \State
   {\it Mark.} Compute a triangulation $\tri\in\mathbb{T}$ with 
    $ \mu^2(\tri)\leq \rho_B\mu_\ell^2$.
 \State 
   {\it Refine.} Generate the overlay $\tri_{\ell+1}$ of $\tri_\ell$ and $\tri$.
 \EndIf
\EndFor
\Ensure Sequence of triangulations
  $\left(\tri_\ell\right)_{\ell\in\mathbb N_0}$,
  discrete solutions 
 $(p_{\ell},\alpha_{\ell})_{\ell\in\mathbb{N}_0}$ 
 and error estimators $(\lambda_\ell)_{\ell\in\mathbb{N}_0}$ 
 and $(\mu_\ell)_{\ell\in\mathbb{N}_0}$.
\hfill$\blacklozenge$
\end{algorithmic}
\end{algo}

\begin{remark}[separate versus collective marking]
The residual-based error estimator $\sqrt{\lambda^2+\mu^2}$
involves the term $\|\varphi - \Pi_k\varphi\|_{L^2(T)}$ without 
a multiplicative positive power of the mesh-size.
Therefore, the optimality of an adaptive algorithm based on 
collective marking (that is $\kappa=\infty$ and 
$\lambda$ replaced by $\sqrt{\lambda^2+\mu^2}$ in Algorithm~\ref{a:PMPafem}) 
does not follow from the abstract framework
from \cite{CarstensenFeischlPagePraetorius2014}. 
The reduction property (axiom (A2) from
\cite{CarstensenFeischlPagePraetorius2014}),
is not fulfilled.
Algorithm~\ref{a:PMPafem} considered here is based on separate 
marking. In this context, the optimality of the adaptive algorithm (see 
Theorem~\ref{t:PMPoptimalafem}) can be proved with a reduction property that 
only considers $\lambda$.
\end{remark}

\begin{remark}\label{r:PMPB1approx}
The step {\it Mark} in the second case 
($\mu_\ell^2>\kappa\lambda_\ell^2$)
can be realized by the algorithm \texttt{Approx}
from \cite{BinevDahmenDeVore2004,CarstensenRabus}, i.e.,
the thresholding second algorithm \cite{BinevDeVore2004} followed by a completion algorithm.
For this algorithm, the assumption (B1) optimal data approximation, which is 
assumed 
to hold in the following, follows from the axioms (B2) and (SA) from 
Subsection~\ref{ss:PMPaxiomB} \cite{CarstensenRabus}.
For a discussion about other algorithms that realize {\it Mark} in the second 
case, see~\cite{CarstensenRabus}. 
\end{remark}

For $s>0$ and $(p,\alpha,\varphi)\in X\times Y \times H(\ddiv,\Omega)$ define 
\begin{align*}
 \left| (p,\alpha,\varphi)\right|_{\mathcal{A}_s}
   := \sup_{N\in\mathbb{N}_0} N^s 
     & \inf_{\tri\in\mathbb{T}(N)} \Big( 
          \|p-\Pi_{X_h(\tri)} p \|_{L^2(\Omega)} \\ 
    &    + \inf_{\beta_\tri\in Y_h(\tri)} 
  \left\|\Curl(\alpha-\beta_\tri)\right\|_{L^2(\Omega)}
           + \|\varphi - \Pi_{X_h(\tri)} \varphi\|_{L^2(\Omega)}\Big).
\end{align*}

\begin{remark}[pure local approximation class]\label{r:PMPlocalapproxclass}
Since $\Omega$ is assumed to be a Lip\-schitz domain,
all patches in an admissible triangulation $\tri\in\mathbb{T}$
are edge-connected, i.e., for all vertices $z\in\nodes$ and 
triangles $T,K\in\tri$ with $z\in T\cap K$, there exists
$m\in\mathbb{N}_0$ and $K_0,\dots,K_m\in\tri$ with 
$K_0 = T$, $K_m=K$, 
$z\in K_0\cap \dots \cap K_m$ and $K_{j-1}\cap K_{j}\in\edges$
for all $1\leq j\leq m$.
Under this assumption,
\cite[Theorem~3.2]{Veeser2014} shows
\begin{align*}
 \min_{v_h\in P_{k+1}(\tri)\cap H^1(\Omega)} \|\nabla(v-v_h)\|_{L^2(\Omega)}
  \approx \|\nabla v - \Pi_{k}\nabla v\|_{L^2(\Omega)}
 \qquad\text{for all }v\in H^1(\Omega).
\end{align*}
Hence,
\begin{align*}
 \left| (p,\alpha,\varphi)\right|_{\mathcal{A}_s}
  &\approx \left| (p,\alpha,\varphi)\right|_{\mathcal{A}_s'}\\
   &:= \sup_{N\in\mathbb{N}_0} N^s 
      \inf_{\tri\in\mathbb{T}(N)} \Big( 
          \|p-\Pi_{X_h(\tri)} p \|_{L^2(\Omega)}\\
  &\qquad\quad + \left\|\Curl\alpha - \Pi_{X_h(\tri)}\Curl\alpha\right\|_{L^2(\Omega)}
           + \|\varphi - \Pi_{X_h(\tri)} \varphi\|_{L^2(\Omega)}\Big).
 \qedhere
\end{align*}
\end{remark}

In the following, we assume that the following assumption (B1) holds for the 
algorithm used in the step {\it Mark} for $\mu_\ell^2>\kappa\lambda_\ell^2$ 
(see Remark~\ref{r:PMPB1approx}).

\begin{assumption}[(B1) optimal data approximation]\label{as:PMPB1}
Assume that $\left| (p,\alpha,\varphi)\right|_{\mathcal{A}_\sigma}$ is finite.
Given a tolerance $\mathrm{Tol}$, 
the algorithm used in {\it Mark} 
in the second case ($\mu_\ell^2>\kappa\lambda_\ell^2$) 
in Algorithm~\ref{a:PMPafem}
computes $\tri_\star\in\mathbb{T}$ with 
\begin{equation*}
 \mathrm{card}(\tri_\star) - \mathrm{card}(\tri_0)
  \lesssim \mathrm{Tol}^{-1/(2\sigma)}
 \qquad\text{and}\qquad 
 \mu^2(\tri_\star)\leq \mathrm{Tol}.
 \qedhere
\end{equation*}
\end{assumption}

The following theorem states optimal convergence rates of 
Algorithm~\ref{a:PMPafem}.

\begin{theorem}[optimal convergence rates of AFEM]\label{t:PMPoptimalafem}
For $0<\rho_B<1$ and sufficiently small $0<\kappa$ and $0<\theta <1$, 
Algorithm~\ref{a:PMPafem}
computes sequences of triangulations $(\tri_\ell)_{\ell\in\mathbb{N}}$ 
and discrete solutions $(p_\ell,\alpha_\ell)_{\ell\in\mathbb{N}}$ 
for the right-hand side $\varphi$
of optimal rate of convergence in the sense that 
\begin{align*}
 (\mathrm{card}(\tri_\ell) - \mathrm{card}(\tri_0))^s 
  \Big(\|p-p_\ell\|_{L^2(\Omega)} + 
 \left\|\Curl(\alpha-\alpha_\ell)\right\|_{L^2(\Omega)}\Big)
   \lesssim \left| (p,\alpha,\varphi)\right|_{\mathcal{A}_s}.
\end{align*}
\end{theorem}

The proof follows from the abstract framework
of \cite{CarstensenRabus}, which employs the bounded 
overhead \cite{BinevDahmenDeVore2004} of the newest-vertex bisection,
under the assumptions (A1)--(A4) and (B2) and (SA)
which are proved in 
Subsections~\ref{ss:PMPstabilityreduction}--\ref{ss:PMPaxiomB}.

\subsection{(A1) stability and (A2) reduction}\label{ss:PMPstabilityreduction}

The following two theorems follow from the structure of $\lambda$.

\begin{theorem}[stability]
Let $\tri_\star$ be an admissible refinement of $\tri$
and $\mathcal{M}\subseteq\tri\cap\tri_\star$.
Let $(p_{\tri_\star},\alpha_{\tri_\star})\in X_h(\tri_\star)\times 
Y_h(\tri_\star)$
and $(p_{\tri},\alpha_{\tri})\in X_h(\tri)\times Y_h(\tri)$
be the respective discrete solutions to \eqref{e:PMPdP}.
Then,
\begin{align*}
  \lvert \lambda(\tri_\star,\mathcal{M}) - \lambda(\tri,\mathcal{M})\rvert
   \lesssim \| p_{\tri_\star} - p_\tri\|_{L^2(\Omega)}.
\end{align*}
\end{theorem}

\begin{proof}
This follows with triangle inequalities,
inverse inequalities and the trace inequality from \cite[p.~282]{BrennerScott08}
as in \cite[Proposition~3.3]{CasconKreuzerNochettoSiebert2008}.
\end{proof}

\begin{theorem}[reduction]
Let $\tri_\star$ be an admissible refinement of $\tri$.
Then there exists $0<\rho_2< 1$ and $\Lambda_2<\infty$ such that
\begin{align*}
 \lambda^2(\tri_\star,\tri_\star\setminus\tri)
  \leq \rho_2\lambda^2(\tri,\tri\setminus\tri_\star)
   + \Lambda_2 \|p_{\tri_\star}- p_\tri\|_{L^2(\Omega)}^2.
\end{align*}
\end{theorem}

\begin{proof}
This follows with a triangle inequality and the mesh-size 
reduction property $h_{\tri_\star}^2\vert_T\leq h_\tri^2\vert_T/2$
for all $T\in\tri_\star\setminus\tri$
as in \cite[Corollary~3.4]{CasconKreuzerNochettoSiebert2008}.
\end{proof}

\subsection{(A4) discrete reliability}\label{ss:PMPdrel}

The following theorem proves discrete reliability, i.e., 
the difference between two discrete solutions is bounded by the error 
estimators on refined triangles only.

\begin{theorem}[discrete reliability]\label{t:PMPdrel}
 Let $\tri_\star$ be an admissible refinement of $\tri$ with 
respective discrete solutions 
$(p_{\tri_\star},\alpha_{\tri_\star})\in X_h(\tri_\star)\times Y_h(\tri_\star)$
and $(p_{\tri},\alpha_{\tri})\in X_h(\tri)\times Y_h(\tri)$.
Then,
\begin{align*}
 \|p_\tri - p_{\tri_\star}\|_{L^2(\Omega)}^2 
  + \left\|\Curl(\alpha_\tri - \alpha_{\tri_\star})\right\|_{L^2(\Omega)}^2
 \lesssim \lambda^2(\tri,\tri\setminus\tri_\star) + 
\mu^2(\tri,\tri\setminus\tri_\star).
\end{align*}
\end{theorem}

\begin{proof}
Recall the definition of $W_h(\tri_\star)$ from~\eqref{e:PMPMh}.
Since $p_\tri-p_{\tri_\star}\in X_h(\tri_\star)$, there exist
$\sigma_{\tri_\star}\in W_h(\tri_\star)$ and $r_{\tri_\star}\in Y_h(\tri_\star)$
with $p_\tri-p_{\tri_\star} = \sigma_{\tri_\star} + \Curl r_{\tri_\star}$.
Since $W_h(\tri_\star)\bot_{L^2(\Omega)} \Curl Y_h(\tri_\star)$,
\begin{align*}
 \|\sigma_{\tri_\star}\|_{L^2(\Omega)}^2 
   + \left\|\Curl r_{\tri_\star}\right\|_{L^2(\Omega)}^2 = 
\|p_\tri-p_{\tri_\star}\|_{L^2(\Omega)}^2.
\end{align*}
The orthogonality furthermore implies that
the discrete error can be split as
\begin{align*}
 \|p_\tri-p_{\tri_\star}\|_{L^2(\Omega)}^2
 = (p_\tri-p_{\tri_\star},\sigma_{\tri_\star})_{L^2(\Omega)}
   + (p_\tri-p_{\tri_\star}, \Curl r_{\tri_\star})_{L^2(\Omega)}.
\end{align*}
The projection property, Lemma~\ref{l:PMPintegralmean},
proves $\Pi_{X_h(\tri)}\sigma_{\tri_\star}\in W_h(\tri)$.
Hence, problem~\eqref{e:PMPdP} implies that the first term of the right-hand 
side equals 
\begin{align*}
 (p_\tri-p_{\tri_\star},\sigma_{\tri_\star})_{L^2(\Omega)}
 = (\Pi_{X_h(\tri)}\varphi - \varphi,\sigma_{\tri_\star})_{L^2(\Omega)}
 = (\Pi_{X_h(\tri)}\varphi - \Pi_{X_h(\tri_\star)}\varphi,
      \sigma_{\tri_\star})_{L^2(\Omega)}.
\end{align*}
For any triangle $T\in\tri\cap\tri_\star$,
it holds $(\Pi_{X_h(\tri)}\varphi - \Pi_{X_h(\tri_\star)}\varphi)\vert_T=0$.
Therefore,
\begin{align*}
 (\Pi_{X_h(\tri)}\varphi - \Pi_{X_h(\tri_\star)}\varphi,
      \sigma_{\tri_\star})_{L^2(\Omega)}
  \leq \|\Pi_{X_h(\tri)}\varphi 
       - \Pi_{X_h(\tri_\star)}\varphi\|_{\tri\setminus\tri_\star}
     \; \|\sigma_{\tri_\star}\|_{L^2(\Omega)}.
\end{align*}
Since $\tri_\star$ is a refinement of $\tri$, 
it holds
\begin{align*}
 \|\Pi_{X_h(\tri)}\varphi 
       - \Pi_{X_h(\tri_\star)}\varphi\|_{\tri\setminus\tri_\star}
 = \|\Pi_{X_h(\tri_\star)}(\Pi_{X_h(\tri)}\varphi 
       - \varphi)\|_{\tri\setminus\tri_\star}
 \leq  \|\varphi - \Pi_{X_h(\tri)}\varphi\|_{\tri\setminus\tri_\star}.
\end{align*}

Let $r_\tri\in Y_h(\tri)$ denote 
the quasi interpolant from \cite{ScottZhang1990}
of $r_{\tri_\star}$ which satisfies the approximation 
and stability properties
\begin{align*}
  \|h_\tri^{-1} (r_{\tri_\star}-r_\tri)\|_{L^2(\Omega)} 
  + \left\|\Curl(r_{\tri_\star}-r_\tri)\right\|_{L^2(\Omega)} 
  \lesssim \left\|\Curl r_{\tri_\star}\right\|_{L^2(\Omega)}
\end{align*}
and $(r_\tri)\vert_E = (r_{\tri_\star})\vert_E$ for all 
edges $E\in \edges(\tri)\cap\edges(\tri_\star)$.
Since $p_\tri\in W_h(\tri)$ and 
$p_{\tri_\star}\in W_h(\tri_\star)$, 
\begin{align}\label{e:PMPdrelproofNCerror}
 (p_\tri-p_{\tri_\star}, \Curl r_{\tri_\star})_{L^2(\Omega)}
  = (p_\tri, \Curl (r_{\tri_\star}-r_\tri))_{L^2(\Omega)}.
\end{align}
An integration by parts leads to 
\begin{align*}
 (p_\tri, \Curl (r_{\tri_\star}-r_\tri))_{L^2(\Omega)}
  &= -(\curl_\NC p_\tri, r_{\tri_\star}-r_\tri)_{L^2(\Omega)}\\
  &\qquad\qquad + \sum_{E\in\mathcal{E}(\tri)} 
         \int_E [p_\tri\cdot\tau_E]_E (r_{\tri_\star}-r_\tri)\,ds.
\end{align*}
For a triangle $T\in \tri\cap\tri_\star$, any edge 
$E\in\edges(T)$ satisfies $E\in\edges(\tri)\cap\edges(\tri_\star)$.
Hence, $(r_\tri)\vert_T = (r_{\tri_\star})\vert_T$
for all $T\in\tri\cap\tri_\star$.
This, the Cauchy inequality and the approximation and 
stability properties of the quasi interpolant lead to 
\begin{align*}
 -(\curl_\NC p_\tri, r_{\tri_\star}-r_\tri)_{L^2(\Omega)}
 \lesssim \|h_\tri \curl_\NC p_\tri\|_{\tri\setminus\tri_\star}
    \left\|\Curl r_{\tri_\star}\right\|_{L^2(\Omega)}.
\end{align*}
Since $(r_\tri)\vert_E = (r_{\tri_\star})\vert_E$ for all 
edges $E\in \edges(\tri)\cap\edges(\tri_\star)$,
the approximation and stability properties of the quasi interpolant
and the trace inequality \cite[p.~282]{BrennerScott08} lead 
to 
\begin{equation}\label{e:proofdrel2}
\begin{aligned}
 &\sum_{E\in\mathcal{E}} 
         \int_E [p_\tri\cdot\tau_E]_E (r_{\tri_\star}-r_\tri)\,ds\\
 &\qquad\qquad\qquad\qquad
   \lesssim \sqrt{\sum_{E\in\edges(\tri)\setminus\edges(\tri_\star)}
       h_T \|[p_\tri\cdot\tau_E]_E\|_{L^2(E)}^2}
   \left\|\Curl r_{\tri_\star}\right\|_{L^2(\Omega)}.
\end{aligned}
\end{equation}
The combination of the previous displayed inequalities
yields 
\begin{align*}
 \|p_\tri - p_{\tri_\star}\|_{L^2(\Omega)}^2 
  \lesssim \lambda^2(\tri,\tri\setminus\tri_\star) 
+ \mu^2(\tri,\tri\setminus\tri_\star).
\end{align*}
Since $\Curl\alpha_\tri = \Pi_{X_h(\tri)}\varphi - p_\tri$ and 
$\Curl\alpha_{\tri_\star} = \Pi_{X_h(\tri_\star)}\varphi - p_{\tri_\star}$,
the triangle inequality yields the assertion.
\end{proof}

The discrete reliability of Theorem~\ref{t:PMPdrel} together with the 
convergence of the discretization proves reliability of the 
residual-based error estimator. This is summarized in the following 
proposition.

\begin{proposition}[efficiency and reliability of the residual-based error 
estimator]
\label{p:effrelResidualEst}
Let $(p,\alpha)\in X\times Y$ and $(p_h,\alpha_h)\in X_h(\tri)\times Y_h(\tri)$ 
be the solutions 
to \eqref{e:PMPmixedproblem} and \eqref{e:PMPdP} for some $\tri\in\mathbb{T}$.
There exist constants $C_\mathrm{eff}, C_\mathrm{rel}>0$ with 
 \begin{align*}
    C_\mathrm{eff}^{-2}(\lambda^2(\tri,\tri)+\mu^2(\tri))
   & \leq \|p-p_h\|_{L^2(\Omega)}^2 + 
\left\|\Curl(\alpha-\alpha_h)\right\|_{L^2(\Omega)}^2\\
   & \leq C_\mathrm{rel}^2 (\lambda^2(\tri,\tri)+\mu^2(\tri)).
 \end{align*}
\end{proposition}

\begin{proof}
The a~priori error estimate from Theorem~\ref{t:PMPbestapprox} implies the 
convergence of the discrete solutions. This and Theorem~\ref{t:PMPdrel} proves 
the reliability. The efficiency follows from the standard bubble function 
technique~\cite{Verfuerth96}.
\end{proof}

\subsection{(A3) quasi-orthogonality}\label{ss:PMPquasiorthogonality}

The following theorem proves quasi-orthogonality of the 
discretization~\eqref{e:PMPdP}.

\begin{theorem}[general quasi-orthogonality]\label{t:PMPquasiorthogonality}
Let $(\tri_j\mid j\in\mathbb{N})$ be some sequence of triangulations with discrete solutions 
$(p_{j},\alpha_{j})\in X_h(\tri_j)\times Y_h(\tri_j)$ to \eqref{e:PMPdP}.
Let $\ell\in\mathbb{N}$. Then,
\begin{align*}
 \sum_{j=\ell}^\infty \Big( \|p_j-p_{j-1}\|_{L^2(\Omega)}^2 
          + \left\|\Curl(\alpha_j-\alpha_{j-1})\right\|_{L^2(\Omega)}^2\Big)
    \lesssim \lambda_{\ell-1}^2 + \mu_{\ell-1}^2.
\end{align*}
\end{theorem}

\begin{proof}
The projection property, Lemma~\ref{l:PMPintegralmean}, proves 
$\Pi_{X_h(\tri_{j-1})} p_j\in W_h(\tri_{j-1})$ with $W_h(\tri_{j-1})$ 
from~\eqref{e:PMPMh}. Hence, problem~\eqref{e:PMPdP} leads to
\begin{align*}
 (p_{j-1},p_j - p_{j-1})_{L^2(\Omega)} 
    &= (\varphi,\Pi_{X_h(\tri_{j-1})} p_j - p_{j-1})_{L^2(\Omega)},\\
 (p_j,p_j - p_{j-1})_{L^2(\Omega)} 
     &= (\varphi,p_j) - (\varphi,\Pi_{X_h(\tri_{j-1})} p_j)_{L^2(\Omega)}.
\end{align*}
The subtraction of these two equations and an index shift leads, for any 
$M\in\mathbb{N}$ with $M>\ell$, to 
\begin{equation}\label{e:PMPquasiorthogonalityProof1}
\begin{aligned}
 &\sum_{j=\ell}^M \|p_j- p_{j-1}\|_{L^2(\Omega)}^2
   = \sum_{j=\ell}^M 
    (\varphi, p_j - \Pi_{X_h(\tri_{j-1})} p_j)_{L^2(\Omega)} \\
  &\qquad\qquad\qquad\qquad\qquad\qquad
          - \sum_{j=\ell}^M (\varphi,\Pi_{X_h(\tri_{j-1})} p_j)_{L^2(\Omega)}
          + \sum_{j=\ell-1}^{M-1} (\varphi,p_{j})_{L^2(\Omega)}\\
  &\qquad\qquad\qquad
    = (\varphi,p_{\ell-1} - p_M)_{L^2(\Omega)} 
   + 2\sum_{j=\ell}^{M} 
      (\varphi, p_j - \Pi_{X_h(\tri_{j-1})} p_j)_{L^2(\Omega)} .
\end{aligned}
\end{equation}
Since $p_j - \Pi_{X_h(\tri_{j-1})} p_j\in X_h(\tri_j)$ is $L^2$-orthogonal to 
$X_h(\tri_{j-1})$, a Cauchy and a weighted Young inequality
imply
\begin{equation}\label{e:PMPquasiorthogonalityProof6}
\begin{aligned}
 &2\sum_{j=\ell}^{M} (\varphi, p_j  - \Pi_{X_h(\tri_{j-1})} p_j)_{L^2(\Omega)}\\
 &\quad= 2\sum_{j=\ell}^{M} 
   (\Pi_{X_h(\tri_j)}\varphi-\Pi_{X_h(\tri_{j-1})}\varphi, 
                     p_j - \Pi_{X_h(\tri_{j-1})} p_j)_{L^2(\Omega)}\\
 &\quad\leq 2 \sum_{j=\ell}^{M} 
     \|\Pi_{X_h(\tri_j)}\varphi-\Pi_{X_h(\tri_{j-1})}\varphi\|_{L^2(\Omega)}^2
   + \frac{1}{2} \sum_{j=\ell}^{M} 
         \|p_j - \Pi_{X_h(\tri_{j-1})} p_j\|_{L^2(\Omega)}^2.
\end{aligned}
\end{equation}
The orthogonality 
$\Pi_{X_h(\tri_j)}\varphi-\Pi_{X_h(\tri_{j-m})}\varphi\bot_{L^2(\Omega)} 
X_h(\tri_{j-m})$ for all $0\leq m\leq j$ proves
\begin{align}\label{e:PMPquasiorthogonalityProof3}
 \sum_{j=\ell}^{M}  
   \|\Pi_{X_h(\tri_j)}\varphi-\Pi_{X_h(\tri_{j-1})}\varphi\|_{L^2(\Omega)}^2 
   = \|\Pi_{X_h(\tri_M)} \varphi - 
                         \Pi_{X_h(\tri_{\ell-1})}\varphi\|_{L^2(\Omega)}^2.
\end{align}
The definition of $\mu_\ell$ yields 
\begin{equation}\label{e:PMPquasiorthogonalityProof4}
\begin{aligned}
 \|\Pi_{X_h(\tri_M)} \varphi - \Pi_{X_h(\tri_{\ell-1})}\varphi\|_{L^2(\Omega)}
 & = \|\Pi_{X_h(\tri_M)} (\varphi - \Pi_{X_h(\tri_{\ell-1})}\varphi)\|_{L^2(\Omega)}
             \\
  & \leq \mu_{\ell-1} .
\end{aligned}
\end{equation}
The combination of 
\eqref{e:PMPquasiorthogonalityProof1}--\eqref{e:PMPquasiorthogonalityProof4} and 
$\|p_j-\Pi_{X_h(\tri_{j-1})}p_j\|_{L^2(\Omega)}\leq\|p_j-p_{j-1}\|_{
L^2(\Omega)}$ leads to 
\begin{equation}\label{e:PMPquasiorthogonalityProof2}
\begin{aligned}
 \frac{1}{2}\sum_{j=\ell}^M \|p_j- p_{j-1}\|_{L^2(\Omega)}^2
   &\leq 2\mu_{\ell-1}^2
     + (\varphi,p_{\ell-1} - p_M)_{L^2(\Omega)} .
\end{aligned}
\end{equation}
The combination of the arguments of 
\eqref{e:PMPdrelproofNCerror}--\eqref{e:proofdrel2}
proves
\begin{align}\label{e:proofquasiorth1}
 (\Curl(\alpha_{M}-\alpha_{\ell-1}),p_{\ell-1})_{L^2(\Omega)}
  \lesssim \lambda_{\ell-1} 
    \left\|\Curl(\alpha_{M}-\alpha_{\ell-1})\right\|_{L^2(\Omega)}
\end{align}
This, the discrete problem \eqref{e:PMPdP}, and 
the discrete reliability  
$\left\|\Curl(\alpha_{M}-\alpha_{\ell-1})\right\|_{L^2(\Omega)}\lesssim 
\lambda_{\ell-1}+\mu_{\ell-1}$ from Theorem~\ref{t:PMPdrel}
lead to 
\begin{equation*}
\begin{aligned}
 &(p_{\ell-1}- p_{M},\Pi_{X_h(\tri_{\ell-1})}\varphi)_{L^2(\Omega)}
  = (p_{\ell-1}-p_{M},p_{\ell-1} + \Curl\alpha_{\ell-1})_{L^2(\Omega)}\\
 &\qquad\qquad = (p_{\ell-1}-p_{M},p_{\ell-1})_{L^2(\Omega)}
  = (\Curl(\alpha_{M}-\alpha_{\ell-1}),p_{\ell-1})_{L^2(\Omega)}\\
 & \qquad\qquad\lesssim \lambda_{\ell-1} 
    \left\|\Curl(\alpha_{M}-\alpha_{\ell-1})\right\|_{L^2(\Omega)}
  \lesssim (\lambda_{\ell-1}+\mu_{\ell-1})^2.
\end{aligned}
\end{equation*}
This and a further application of Theorem~\ref{t:PMPdrel} 
leads to 
\begin{equation}\label{e:PMPquasiorthogonalityProof8}
\begin{aligned}
 &(\varphi,p_{\ell-1}-p_{M})_{L^2(\Omega)} \\
 &\qquad\;
  = (\varphi-\Pi_{X_h(\tri_{\ell-1})}\varphi,p_{\ell-1}-p_{M})_{L^2(\Omega)} + 
   (p_{\ell-1}-p_{M},\Pi_{X_h(\tri_{\ell-1})}\varphi)_{L^2(\Omega)}\\
 &\qquad\;
  \lesssim \|\varphi-\Pi_{X_h(\tri_{\ell-1})}\varphi\|_{L^2(\Omega)}\;
     \|p_{\ell-1}-p_{M}\|_{L^2(\Omega)} 
    + (\lambda_{\ell-1}+\mu_{\ell-1})_{L^2(\Omega)}^2\\
  &\qquad\;
   \lesssim (\lambda_{\ell-1}+\mu_{\ell-1})^2.
\end{aligned}
\end{equation}
The combination of \eqref{e:PMPquasiorthogonalityProof2} with 
\eqref{e:PMPquasiorthogonalityProof8} implies
\begin{align}\label{e:PMPquasiorthogonalityProof5}
 \sum_{j=\ell}^M \|p_j- p_{j-1}\|_{L^2(\Omega)}^2
    \lesssim \lambda_{\ell-1}^2+\mu_{\ell-1}^2.
\end{align}
The Young inequality,
the triangle inequality, and $\Curl\alpha_j = \Pi_{X_h(\tri_j)} \varphi - p_j$ 
imply
\begin{align*}
 &\sum_{j=\ell}^M
   \left\|\Curl(\alpha_j - \alpha_{j-1})\right\|_{L^2(\Omega)}^2 \\
 &\qquad\qquad 
   \leq 2 \sum_{j=\ell}^M \|p_j - p_{j-1}\|_{L^2(\Omega)}^2
   + 2 \sum_{j=\ell}^M 
      \|\Pi_{X_h(\tri_j)}\varphi - 
                               \Pi_{X_h(\tri_{j-1})}\varphi\|_{L^2(\Omega)}^2.
\end{align*}
Since $M>\ell$ is arbitrary,
the combination with \eqref{e:PMPquasiorthogonalityProof3}, 
\eqref{e:PMPquasiorthogonalityProof4}, and \eqref{e:PMPquasiorthogonalityProof5}
yields the assertion.
\end{proof}

\subsection{(B) data approximation}\label{ss:PMPaxiomB}

The following theorem together with Assumption~\ref{as:PMPB1} form the axiom (B) 
from~\cite{CarstensenRabus}.

\begin{theorem}[(B2) quasimonotonicity and (SA) sub-additivity]
Any admissible refinement $\tri_\star$ of $\tri$ satisfies 
\begin{align*}
  \mu^2(\tri_\star) \leq \mu^2(\tri)
  \quad\text{and}\quad 
   \sum_{\substack{T\in\tri_\star\\
             T\subseteq K}}
   \mu^2(T) \leq \mu^2(K)
 \qquad\text{for all }K\in \tri.
\end{align*}
\end{theorem}

\begin{proof}
This follows directly from the definition of $\mu$.
\end{proof}

\section{Extension to 3D}\label{s:PMP3D}

This section is devoted to the generalization to 3D. 
Subsection~\ref{ss:PMP3Ddef} defines the novel discretization and comments on 
basic properties, while Subsection~\ref{ss:PMP3Dafem} is devoted to optimal 
convergence rates for the adaptive algorithm.

\subsection{Weak formulation and discretization}\label{ss:PMP3Ddef}

For this section, let $\Omega\subseteq\R^3$ be a 
simply connected, bounded, polygonal Lipschitz domain in $\R^3$.
For the sake of simplicity, we also assume that $\partial\Omega$ is connected
(i.e., $\Omega$ is contractible).
The Curl operator acts on a sufficiently smooth vector field 
$\beta:\Omega\to\R^3$ as $\Curl \beta = \nabla \wedge \beta$ 
with the cross product or vector product $\wedge$.
Let $H(\Curl,\Omega)$ denote the space of all $\beta\in L^2(\Omega;\R^3)$  
with $\Curl \beta\in L^2(\Omega;\R^3)$ for the weak $\Curl$, i.e., 
\begin{align*}
 \int_\Omega v\cdot\Curl\beta \,dx
  = \int_\Omega \beta\cdot\Curl v\,dx
 \qquad\text{for all }v\in C^\infty_c(\Omega;\R^3).
\end{align*}
In contrast to the two-dimensional case, $H(\Curl,\Omega)\neq 
H^1(\Omega;\R^3)$. 
The Helmholtz decomposition in 3D reads 
\begin{align}\label{e:PMP3DHelmholtzdec}
 L^2(\Omega;\R^3) 
   = \nabla H^1_0(\Omega) \oplus \Curl H(\Curl,\Omega)
\end{align}
and the sum is $L^2$ orthogonal.
It is a consequence of the identity 
\begin{align*}
 \{r\in H(\ddiv,\Omega)\mid \ddiv r=0\}
  = \Curl H(\Curl,\Omega)
\end{align*}
in the De Rham complex \cite{BoffiBrezziFortin2013}.

Let $\varphi\in H(\ddiv,\Omega)$ with $-\ddiv\varphi=f$. Then
the Poisson problem \eqref{e:PMPstrongform} is equivalent to the problem: 
Find $(p,\alpha)\in L^2(\Omega;\R^3)\times H(\Curl,\Omega)$ with 
\begin{equation}\label{e:PMP3Dsemiproblem}
\begin{aligned}
 (p,q)_{L^2(\Omega)} + (q,\Curl \alpha)_{L^2(\Omega)} &= 
(\varphi,q)_{L^2(\Omega)}
 &&\text{ for all }q\in L^2(\Omega;\mathbb{R}^3),\\
 (p,\Curl \beta)_{L^2(\Omega)} &= 0 
 &&\text{ for all }\beta\in H(\Curl,\Omega).
\end{aligned}
\end{equation}
In contrast to the two-dimensional case, the operator 
$\Curl:H(\Curl,\Omega)\to L^2(\Omega;\R^3)$ has a non-trivial kernel. 
Classical results \cite{Rudin1976}
characterize this kernel as $\nabla H^1(\Omega)$.
To enforce uniqueness, we can reformulate \eqref{e:PMP3Dsemiproblem} as follows.
Seek $(p,\alpha,w)\in 
L^2(\Omega;\R^3)\times H(\Curl,\Omega) \times (H^1(\Omega)\cap L^2_0(\Omega))$ with 
\begin{equation*}
\begin{aligned}
 (p,q)_{L^2(\Omega)} + (q,\Curl \alpha)_{L^2(\Omega)}
  &= (\varphi,q)_{L^2(\Omega)}
 &&\text{ for all }q\in L^2(\Omega;\mathbb{R}^3),\\
 (p,\Curl \beta)_{L^2(\Omega)} + (\beta,\nabla w)_{L^2(\Omega)} 
  &= 0 
 &&\text{ for all }\beta\in H(\Curl,\Omega),\\
 (\alpha,\nabla v)_{L^2(\Omega)} &= 0 
 && \text{ for all }v\in (H^1(\Omega)\cap L^2_0(\Omega)).
\end{aligned}
\end{equation*}
Note that $\{\beta\in H(\Curl,\Omega)\mid \Curl\beta=0\}
=\nabla H^1(\Omega)$ implies $w=0$.

Standard finite element spaces to discretize $H(\Curl,\Omega)$ in 3D
are the N\'ed\'elec finite element spaces \cite{Nedelec1980, Nedelec1986} 
(also called edge elements)
which are known from the context 
of Maxwell's equations.
Let $\tri$ be a regular triangulation of $\Omega$ in tetrahedra in the sense 
of~\cite{Ciarlet1978}.
The spaces of first kind N\'ed\'elec finite elements read
\begin{align*}
 Y_{\mathrm{N},k}(T) &:= P_k(T;\R^3) + (x \wedge P_k(T;\R^3)),\\
 Y_{\mathrm{N},k}(\tri)&:=\{\beta_h\in H(\Curl,\Omega)\mid 
      \forall T\in\tri:\; \beta_h\vert_T\in Y_{\mathrm{N},k}(T)\}.
\end{align*}
Let $X_h(\tri):=P_k(\tri;\R^3)$.
Since $\Curl Y_{\mathrm{N},k}(\tri)\subseteq X_h(\tri)$, a generalization of 
\eqref{e:PMPdP} to 3D seeks 
$(p_h,\alpha_h)\in X_h(\tri)\times Y_{\mathrm{N},k}(\tri)$
with 
\begin{equation}\label{e:PMP3DdProb}
\begin{aligned}
  (p_h,q_h)_{L^2(\Omega)} + (q_h,\Curl \alpha_h)_{L^2(\Omega)} &= 
(\varphi,q_h)_{L^2(\Omega)}
 &&\text{ for all }q_h\in X_h(\tri),\\
 (p_h,\Curl \beta_h)_{L^2(\Omega)} &= 0 
 &&\text{ for all }\beta_h\in Y_{\mathrm{N},k}(\tri).
\end{aligned}
\end{equation}
The discrete exact sequence \cite{BoffiBrezziFortin2013}
implies that the elements in $Y_{\mathrm{N},k}(\tri)$ with 
vanishing Curl are exactly the gradients of functions in 
$U_h(\tri):=P_{k+1}(\tri)\cap H^1(\Omega)\cap L^2_0(\Omega)$. Therefore, 
the uniqueness in~\eqref{e:PMP3DdProb} can be obtained in the following 
formulation. Seek 
$(p_h,\alpha_h,w_h)\in X_h(\tri)\times Y_{\mathrm{N},k}(\tri)\times U_h(\tri)$
with 
\begin{equation}\label{e:PMP3DdProbfull}
\begin{aligned}
  (p_h,q_h)_{L^2(\Omega)} + (q_h,\Curl \alpha_h)_{L^2(\Omega)} &= 
(\varphi,q_h)_{L^2(\Omega)}
 &&\text{ for all }q_h\in X_h(\tri),\\
 (p_h,\Curl \beta_h)_{L^2(\Omega)} + (\beta_h,\nabla w_h)_{L^2(\Omega)} &= 0 
 &&\text{ for all }\beta_h\in Y_{\mathrm{N},k}(\tri),\\
  (\alpha_h,\nabla v_h)_{L^2(\Omega)} &= 0 
 &&\text{ for all } v_h\in U_h(\tri).
\end{aligned}
\end{equation}
Note that $\nabla U_h(\tri)$ is the kernel of 
$\Curl:Y_{\mathrm{N},k}(\tri)\to P_k(\tri;\R^3)$ and so \eqref{e:PMP3DdProbfull}
implies $w_h=0$.
This variable is introduced in order that \eqref{e:PMP3DdProbfull} has
the form of a standard mixed system.
The discrete Helmholtz decomposition of 
\cite[Lemma~5.4]{AlonsoRodriguezHiptmairValli2004} proves that for the lowest 
order discretization $k=0$, $p_h$ is a 
Crouzeix-Raviart function and so \eqref{e:PMP3DdProbfull} can be seen as a 
generalization of the non-conforming Crouzeix-Raviart FEM to 
higher polynomial degrees.

The inf-sup condition follows from $\nabla U_h(\tri)
\subseteq Y_{\mathrm{N},k}(\tri)$ and $\Curl Y_{\mathrm{N},k}(\tri)\subseteq 
X_h(\tri)$. This and the conformity of the method lead to the 
best-approximation result 
\begin{align*}
   & \|p-p_h\|_{L^2(\Omega)} 
    + \left\|\Curl(\alpha-\alpha_h)\right\|_{L^2(\Omega)}
    + \|\nabla(w-w_h)\|_{L^2(\Omega)}\\
   &\qquad\qquad
   \lesssim \Big(\min_{q_h\in X_h(\tri)} 
  \|p-q_h\|_{L^2(\Omega)}
      +\min_{\beta_h\in Y_{\mathrm{N},k}(\tri)} 
	\left\|\Curl (\alpha-\beta_h)\right\|_{L^2(\Omega)}\\
   &\qquad\qquad\qquad\qquad\qquad\qquad\qquad\qquad
     + \min_{s_h\in U_h(\tri)}
         \|\nabla(w-s_h)\|_{L^2(\Omega)}\Big).
\end{align*}
Since $w=w_h=0$, this is equivalent to 
\begin{align*}
   & \|p-p_h\|_{L^2(\Omega)} 
    + \left\|\Curl(\alpha-\alpha_h)\right\|_{L^2(\Omega)}\\
   &\qquad\qquad
   \lesssim \Big(\min_{q_h\in X_h(\tri)} 
  \|p-q_h\|_{L^2(\Omega)}
      +\min_{\beta_h\in Y_{\mathrm{N},k}(\tri)} 
	\left\|\Curl (\alpha-\beta_h)\right\|_{L^2(\Omega)}\Big).
\end{align*}
The following proposition states a projection property similar to 
Lemma~\ref{l:PMPintegralmean} for the two-dimensional case.
To this end, define 
\begin{align*}
 Z_h(\tri)&:=\{ \beta_h\in Y_{\mathrm{N},k}(\tri)\mid
   \forall v_h\in U_h(\tri):\;
   (\beta_h,\nabla v_h)_{L^2(\Omega)}=0)\},\\
 W_h(\tri)&:= \{q_h\in X_h(\tri) \mid 
  \forall \beta_h\in Z_h(\tri):\;
   (q_h,\Curl\beta_h)_{L^2(\Omega)} =0\}.
\end{align*}
Since $\nabla U_h(\tri)$ is the kernel of 
$\Curl:Y_{\mathrm{N},k}(\tri)\to X_h(\tri)$, it holds 
\begin{align*}
  \Curl Y_{\mathrm{N},k}(\tri) = \Curl Z_h(\tri).
\end{align*}
This implies 
\begin{align*}
 W_h(\tri)&=\{q_h\in X_h(\tri) \mid 
  \forall \beta_h\in Y_{\mathrm{N},k}(\tri):\;(q_h,\Curl\beta_h)_{L^2(\Omega)} =0\}.
\end{align*}

\begin{lemma}[projection property]\label{l:PMP3Dintegralmean}
Let $q\in L^2(\Omega;\R^3)$ with $(q,\Curl\beta)_{L^2(\Omega)}=0$
for all $\beta\in H(\Curl,\Omega)$
(that means that $q$ is a gradient of a $H^1_0(\Omega)$ function). 
Then $\Pi_{X_h(\tri)} q\in W_h(\tri)$.
If $\tri_\star$ is an admissible refinement of $\tri$, 
then $\Pi_{X_h(\tri)}W_h(\tri_\star)\subseteq W_h(\tri) $.
\end{lemma}

\begin{proof}
Since $\Curl Y_{\mathrm{N},k}(\tri)\subseteq X_h(\tri)$ and 
$Y_{\mathrm{N},k}(\tri)\subseteq 
H(\Curl,\Omega)$, the assertion follows with the arguments in the proof of 
Lemma~\ref{l:PMPintegralmean}.
\end{proof}

\subsection{Adaptive algorithm}\label{ss:PMP3Dafem}

This subsection outlines the proof of optimal convergence rates for
Algorithm~\ref{a:PMPafem} in 3D driven by the error estimators 
$\lambda$ and $\mu$ defined by the local contributions
\begin{align*}
 \lambda^2(\tri_\ell,T)&:=\|h_\tri\Curl_\NC p_h\|_{L^2(T)}^2 
   + h_T\sum_{E\in\edges(T)} \|[p_h\wedge \nu_E]_E\|_{L^2(E)}^2,\\
 \mu^2(T)&:=\|\varphi-\Pi_{X_h(\tri)}\varphi\|_{L^2(T)}^2
\end{align*}
and~\eqref{e:PMPdefmulambda2}. Here, $\edges(T)$ denotes the faces of a 
tetrahedron $T\in\tri$ and $h_\tri\in P_0(\tri)$ denotes the piecewise
constant mesh-size function defined by 
$h_\tri\vert_T:=h_T:=\mathrm{meas}_3(T)^{1/3}$. 
The refinement of triangulations in 
Algorithm~\ref{a:PMPafem} is done by newest-vertex 
bisection~\cite{Stevenson2008}.
Let $\mathbb{T}(N)$ denote the space of admissible triangulations with at most 
$N$ tetrahedra more than $\tri_0$.
As in Subsection~\ref{ss:PMPafemdef}, define the seminorm
\begin{align*}
 \left\vert(p,\alpha,\varphi)\right\vert_{\mathcal{A}_s}
  :=\sup_{N\in\mathbb{N}_0} N^s & \inf_{\tri\in\mathbb{T}(N)}
      \Big( \|p-\Pi_{X_h(\tri)} p \|_{L^2(\Omega)} \\ 
    &    + \inf_{\beta_\tri\in Y_{\mathrm{N},k}(\tri)} 
   \left\|\Curl(\alpha-\beta_\tri)\right\|_{L^2(\Omega)}
           + \|\varphi - \Pi_{X_h(\tri)} \varphi\|_{L^2(\Omega)}\Big).
\end{align*}
Assume that Assumption~\ref{as:PMPB1} holds.
The following theorem states optimal convergence rates for 
Algorithm~\ref{a:PMPafem} for 3D.

\begin{theorem}[optimal convergence rates of AFEM for 3D]
Let $s>0$.
For $0<\rho_B<1$ and sufficiently small $0<\kappa$ and $0<\theta <1$, 
Algorithm~\ref{a:PMPafem}
computes sequences of triangulations $(\tri_\ell)_{\ell\in\mathbb{N}}$ 
and discrete solutions $(p_\ell,\alpha_\ell)_{\ell\in\mathbb{N}}$ 
for the right-hand side $\varphi$
of optimal rate of convergence in the sense that 
\begin{align*}
 (\mathrm{card}(\tri_\ell) - \mathrm{card}(\tri_0))^s 
  \Big(\|p-p_\ell\|_{L^2(\Omega)} + 
 \left\|\Curl(\alpha-\alpha_\ell)\right\|_{L^2(\Omega)}\Big)
   \lesssim \left| (p,\alpha,\varphi)\right|_{\mathcal{A}_s}.
\end{align*}
\end{theorem}

The proof follows as in Section~\ref{s:PMPafem} from (A1)--(A4) and (B) 
from~\cite{CarstensenRabus} and the efficiency of $\lambda$ and $\mu$. 
The proof of efficiency follows with the standard bubble-function 
technique~\cite{Verfuerth96}.
The proofs of the axioms (A1)--(A4) and (B) are outlined in the 
following.

The axioms (A1) stability and (A2) reduction follow as in 
Subsection~\ref{ss:PMPstabilityreduction} with triangle inequalities, inverse 
inequalities, a trace inequality similar to \cite[p.~282]{BrennerScott08}, and 
the mesh-size reduction property $h_{\tri_\star}^3\vert_T\leq 
h_\tri^3\vert_T/2$ for 
all $T\in\tri_\star\setminus\tri$. 
However, for (A3) quasi-orthogonality and (A4) discrete reliability, the 
interpolation operator of~\cite{ScottZhang1990} cannot be applied 
directly to $r_{\tri_\star}\in Y_{\mathrm{N},k}(\tri_\star)$ as done in the proof of 
Theorem~\ref{t:PMPdrel}, because $Y_{\mathrm{N},k}(\tri_\star)\not\subseteq 
H^1(\Omega;\R^3)$. This can be overcome by a
quasi-interpolation based on a quasi-interpolation operator 
from~\cite{Schoeberl2008} and a projection operator 
from~\cite{ZhongChenShuWittumXu2012}. Its properties are summarized in 
the following theorem.

\begin{theorem}[quasi-interpolation]\label{t:PMP3Dquasiinterpol}
Let $\tri_\star$ be an admissible refinement of $\tri$ and define 
$\mathcal{R}(\tri,\tri_\star):=\{T\in\tri\mid \exists 
K_1\in\tri\setminus\tri_\star\exists K_2\in\tri\text{ with }K_1\cap 
K_2\neq\emptyset\text{ and }T\cap K_2\neq\emptyset\}$.
Let $\gamma_{\tri_\star}\in Z_h(\tri_\star)$. Then there exists 
$\gamma_\tri\in Y_{\mathrm{N},k}(\tri)$, $\rho\in H^1(\Omega)$, and $\Phi\in 
H^1(\Omega;\R^3)$ with  
\begin{align*}
 \gamma_{\tri_\star} - \gamma_\tri &= \nabla \rho+\Phi,\\
 (\gamma_{\tri_\star} - \gamma_\tri)\vert_T & =0
  \text{ for all }T\in \tri\setminus \mathcal{R}(\tri,\tri_\star),\\
 \|h_\tri^{-1}\Phi\|_{L^2(\Omega)}+\|\nabla\Phi\|_{L^2(\Omega)}
 &\lesssim
     \left\|\Curl \gamma_{\tri_\star}\right\|_{L^2(\Omega)}.
\end{align*}
\end{theorem}

\begin{proof}
This follows as in the proof of \cite[Theorem~5.3]{ZhongChenShuWittumXu2012}
and with the ellipticity on the discrete 
kernel from \cite[Proposition~4.6]{AmroucheBernardiDaugeGirault1998}.
\end{proof}

The differences between the proof of (A4) discrete reliability and 
the proof of Theorem~\ref{t:PMPdrel} are outlined in the following.
Let 
$(p_{\tri_\star},\alpha_{\tri_\star})\in X_h(\tri_\star)\times 
Z_h(\tri_\star)$ and $(p_\tri,\alpha_\tri)\in X_h(\tri)\times Z_h(\tri)$ 
denote the discrete solutions to~\eqref{e:PMP3DdProb}. As in the proof of 
Theorem~\ref{t:PMPdrel}, let $\sigma_{\tri_\star}\in W_h(\tri_\star)$ and 
$r_{\tri_\star}\in Z_h(\tri_\star)$ such that $p_\tri - 
p_{\tri_\star}=\sigma_{\tri_\star} + \Curl r_{\tri_\star}$. The first term of 
the right-hand side of 
\begin{align*}
 \|p_\tri-p_{\tri_\star}\|^2 
   = (p_\tri-p_{\tri_\star},\sigma_{\tri_\star})_{L^2(\Omega)} 
   + (p_\tri-p_{\tri_\star}, \Curl r_{\tri_\star})_{L^2(\Omega)} 
\end{align*}
is estimated as in the proof of Theorem~\ref{t:PMPdrel}, while for the second 
term, the quasi-interpolant $r_\tri\in Y_{\mathrm{N},k}(\tri)$ of 
$r_{\tri_\star}$ with $r_{\tri_\star}-r_\tri=\nabla\rho+\Phi$ for $\rho\in 
H^1(\Omega)$ and $\Phi\in H^1(\Omega;\R^3)$ from 
Theorem~\ref{t:PMP3Dquasiinterpol} is employed. This 
yields
\begin{align*}
 (p_\tri-p_{\tri_\star},\Curl r_{\tri_\star})_{L^2(\Omega)}
  =  (p_\tri,\Curl(r_{\tri_\star}-r_\tri))_{L^2(\Omega)}
  = (p_\tri,\Curl\Phi)_{L^2(\Omega)}.
\end{align*}
A piecewise integration by parts and the arguments of the proof of 
Theorem~\ref{t:PMPdrel} conclude the proof.
The crucial point is that $\Phi\in H^1(\Omega;\R^3)$ is smooth enough to allow 
for a trace inequality.

The proof of (A3) quasi-orthogonality follows as in the proof of 
Theorem~\ref{t:PMPquasiorthogonality} with the projection property of 
Lemma~\ref{l:PMP3Dintegralmean} and the following modifications 
in~\eqref{e:proofquasiorth1}. 
Since (in the analogue notation as in \eqref{e:proofquasiorth1}) 
$\alpha_{\ell-1}\in 
Z_h(\tri_{\ell-1})\subseteq Y_{\mathrm{N},k}(\tri_M)$, there exists $\gamma_M\in 
Z_h(\tri_M)$ with $\Curl\gamma_M=\Curl\alpha_{\ell-1}$.
Theorem~\ref{t:PMP3Dquasiinterpol} guarantees the existence of 
$\beta_{\ell-1}\in Y_{N,k}(\tri_{\ell-1})$, $\rho\in H^1(\Omega)$ and $\Phi\in 
H^1(\Omega;\R^3)$ with $\alpha_M-\gamma_M-\beta_{\ell-1}=\nabla \rho + \Phi$.
This implies in~\eqref{e:proofquasiorth1} that
\begin{align*}
 (\Curl(\alpha_M-\alpha_{\ell-1}),p_{\ell-1})_{L^2(\Omega)} 
  &=  (\Curl(\alpha_M-\gamma_{M}-\beta_{\ell-1}),p_{\ell-1})_{L^2(\Omega)} \\
 &= (\Curl\Phi,p_{\ell-1})_{L^2(\Omega)}.
\end{align*}
Since $\Phi\in H^1(\Omega;\R^3)$ is smooth enough, a
piecewise integration by parts and the arguments of the proof of 
Theorem~\ref{t:PMPdrel} then prove
\begin{align*}
 (\Curl(\alpha_M-\alpha_{\ell-1}),p_{\ell-1})_{L^2(\Omega)}
  \lesssim (\lambda_{\ell-1}+\mu_{\ell-1}) 
    \left\|\Curl(\alpha_M-\alpha_{\ell-1})\right\|_{L^2(\Omega)} .
\end{align*}
This and the arguments of Theorem~\ref{t:PMPquasiorthogonality} eventually 
prove the quasi-orthogonality.

\section{Numerical experiments}\label{s:PMPnumerics}

This section presents numerical experiments for the 
discretization~\eqref{e:PMPdP} for $k=0,1,2$. 
Subsections~\ref{ss:PMPnumLshapedDbEx}--\ref{ss:PMPnumSingAlpha} compute the 
discrete solutions on sequences of uniformly red-refined triangulations (see 
Figure~\ref{f:PMPredrefinement} for a red-refined triangle) as well as on 
sequences of triangulations created by the adaptive algorithm~\ref{a:PMPafem} 
with bulk parameter $\theta=0.1$ and $\kappa=0.5$ and $\rho=0.75$.
The convergence history plots are logarithmically scaled and display the error 
$\|p-p_h\|_{L^2(\Omega)}$ against the number of degrees of freedom (ndof) of the 
linear system resulting from the Schur complement.
The underlying L-shaped domain 
$\Omega:=(-1,1)^2\setminus ([0,1]\times [-1,0])$ with its initial 
triangulation is depicted in 
Figure~\ref{f:PMPLdomain}.

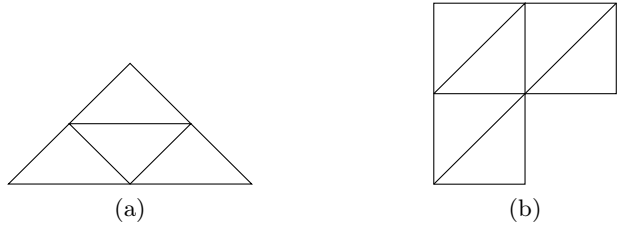
\begin{figure}
\begin{center}
\subfloat[\label{f:PMPredrefinement}]{
  \begin{tikzpicture}[x=0.8cm, y=0.8cm]
    \draw (0,0)--(4,0)--(2,2)--cycle;
    \draw (1,1)--(2,0)--(3,1)--cycle;
  \end{tikzpicture}
}
\hspace{2cm}
\subfloat[\label{f:PMPLdomain}]{
  \begin{tikzpicture}[x=1.2cm, y=1.2cm]
    \draw (0,0)--(1,0)--(1,1)--(-1,1)--(-1,-1)--(0,-1)--cycle;
    \draw (-1,-1)--(1,1);
    \draw (-1,0)--(0,1)--(0,0)--(-1,0);
  \end{tikzpicture}
}
\end{center}
\caption{Red-refined triangle and initial mesh for the L-shaped domain.}
\end{figure}

\subsection{L-shaped domain, I}
\label{ss:PMPnumLshapedDbEx}

The function $u$ given in polar coordinates by
\begin{align*}
 u(r,\phi)=r^{2/3} \sin((2/3)\phi)
\end{align*}
is harmonic.
For the following experiment we choose $\varphi\equiv 0$ and $u_D:=g\,u$ 
with perturbation function $g\in H^2(\Omega)$,
\begin{align*}
  g(x):= \begin{cases}
                 0 & \text{ if }\lvert x\rvert\leq 1/2,\\
                 16 \lvert x\rvert^4 - 64 \lvert x\rvert^3
                 +88 \lvert x\rvert^2-48 \lvert x\rvert+9 
                        & \text{ if }1/2\leq \lvert x\rvert\leq 1,\\
                 1 & \text{ if }\lvert x\rvert\geq 1,
               \end{cases}
\end{align*}
such that 
$g\vert_{\Gamma}=1$ for $\Gamma:=\partial\Omega\setminus(\{0\}\times (-1,0)
\cup (0,1)\times\{0\})$. Since $u\vert_{\partial\Omega\setminus\Gamma}=0$, 
it holds $u_D\vert_{\partial\Omega}=u$. 
Let $B_{1/2}(0):=\{x\in \R^2\mid \lvert x\rvert< 1/2\}$ denote the ball with 
radius $1/2$ and midpoint $(0,0)$. Since $g\vert_{B_{1/2}(0)}=0$ and $u\in 
H^2(\Omega\setminus B_{1/2}(0))$, it holds $u_D\in H^2(\Omega)$.

For non-homogeneous Dirichlet data, the jump $[p_h]_E\cdot\tau_E$ is defined
for boundary edges $E\in\edges$, $E\subseteq \Gamma_D$, with adjacent triangle 
$T_+$ by 
\begin{align*}
 [p_h]_E\cdot\tau_E:= p_h\vert_{T_+}\cdot\tau_E - \nabla u_D\cdot\tau_E.
\end{align*}
The error estimator $\lambda$ is then defined by 
\eqref{e:PMPdefmulambda1}--\eqref{e:PMPdefmulambda2}.
The local data error estimator contributions read
\begin{align*}
  \mu^2(T):=\|(\varphi-\nabla u_D)-\Pi_k(\varphi-\nabla u_D)\|_{L^2(T)}^2.
\end{align*}
The global error estimator $\mu$ is defined by \eqref{e:PMPdefmulambda2}.

The errors and error estimators for the approximation 
$p_h\in P_k(\tri;\R^2)$ of $\nabla u$ for $k=0,1,2$ 
are plotted in Figure~\ref{f:PMPnumLshapedDbEx} against 
the number of degrees of freedom. The errors and error estimators show an 
equivalent behaviour with an overestimation of approximately 10.
Uniform refinement leads to a suboptimal convergence rate 
of $h^{2/3}\approx\texttt{ndof}^{-1/3}$ for $k=0,1,2$. The adaptive refinement 
reproduces 
the optimal convergence rates of $\texttt{ndof}^{-(k+1)/2}$ for $k=0,1,2$.
Figure~\ref{f:PMPnumLshapedDbExTriang} depicts three meshes created by the 
adaptive algorithm for $k=0$, $1$, and $2$ with approximately 1000 degrees of 
freedom. The singularity at the re-entrant corner leads to a strong refinement 
towards $(0,0)$, while the refinement for $k=0,1$ also reflects the behaviour of the 
right-hand side, i.e., one also observes a moderate refinement on the 
circular ring $\{x\in\Omega\mid 1/2\leq \lvert x\rvert\leq 1\}$.
The marking with respect to the data-approximation 
($\mu_\ell^2>\kappa\lambda_\ell^2$ in Algorithm~\ref{a:PMPafem}) is 
applied at the first 7 (resp.\ 5 and 10) levels for $k=0$ (resp.\ $k=1$ and 
$k=2$) and then at approximately every third level.

\begin{figure}
 \begin{center}
 \includegraphics{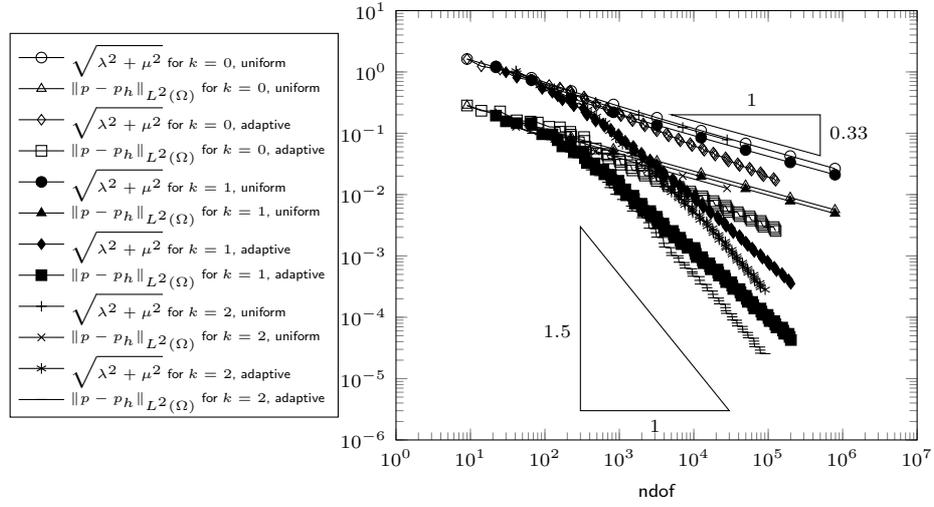}
 \end{center}
 \caption[Errors and error estimators from 
 Subsection~\ref{ss:PMPnumLshapedDbEx}.]{\label{f:PMPnumLshapedDbEx}Errors and error estimators from 
 Subsection~\ref{ss:PMPnumLshapedDbEx}.}
\end{figure}

\begin{figure}
 \begin{center}
  \includegraphics[width=0.32\textwidth]
      {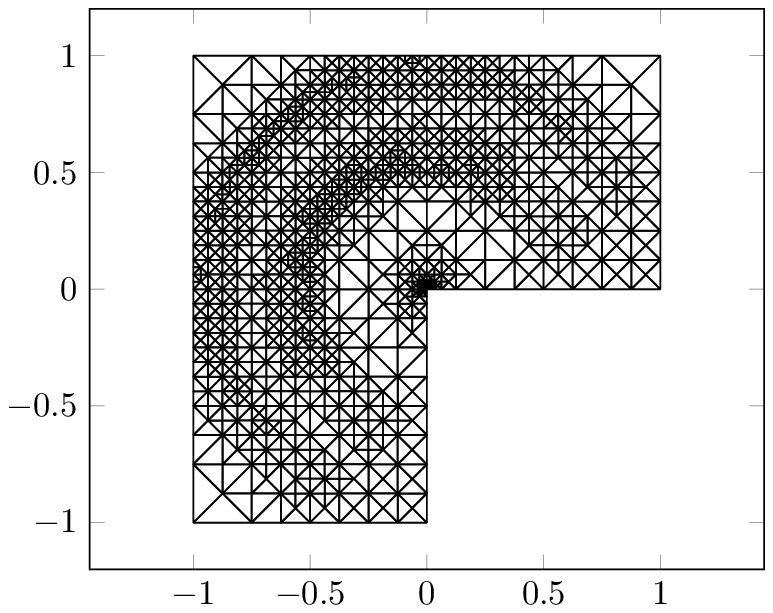}
  \includegraphics[width=0.32\textwidth]
      {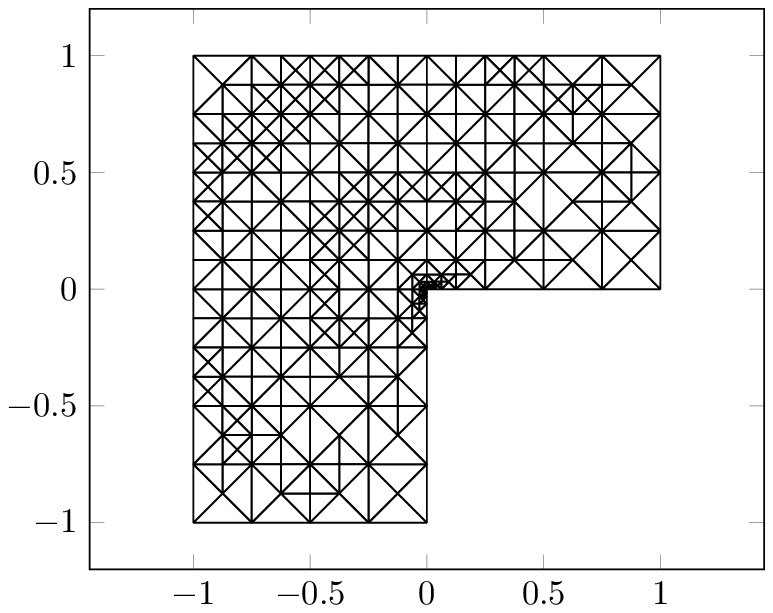} 
  \includegraphics[width=0.32\textwidth]
      {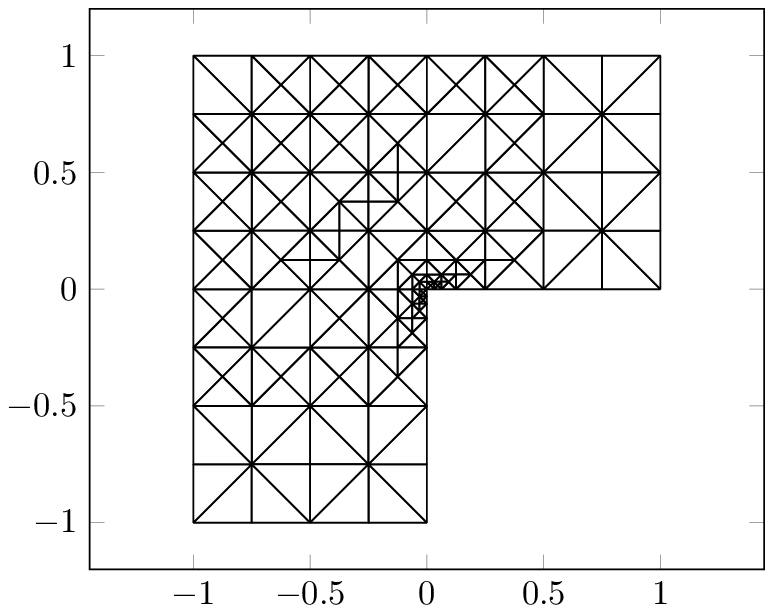}
 \end{center}
  \caption[Adaptively refined triangulations 
 for the experiment from 
Subsection~\ref{ss:PMPnumLshapedDbEx}.]{\label{f:PMPnumLshapedDbExTriang}Adaptively refined triangulations 
for the experiment from 
Subsection~\ref{ss:PMPnumLshapedDbEx}.}
\end{figure}

\subsection{L-shaped domain, II}
\label{ss:PMPnumLshaped}

For $f\equiv -1$ and $u_D\equiv 0$ 
define $\varphi(x,y):=(1/2) (x,y)$ with $-\ddiv\varphi=f$.

The error estimators are plotted against the degrees 
of freedom in Figure~\ref{f:PMPnumLshaped} for $k=0,1,2$. The error 
estimators 
show for $k=0,1,2$ a suboptimal convergence rate of 
$h^{2/3}\approx\texttt{ndof}^{-1/3}$ for uniform refinement. The adaptive 
algorithm~\ref{a:PMPafem} recovers the optimal convergence rate of 
$\texttt{ndof}^{-(k+1)/2}$.
Adaptively refined meshes are depicted in Figure~\ref{f:PMPnumLshapedTriang} 
for approximately 1000 degrees of freedom. The strong refinement towards 
the singularity at the re-entrant corner is clearly visible.
The smoothness of $\varphi\in P_1(\Omega;\R^2)$ implies that the 
data-approximation error estimator $\mu_\ell$ vanishes on all 
triangulations for $k=1,2$. For $k=0$, $\mu_\ell$ does not vanish, 
nevertheless, since $\mu_\ell^2\leq\kappa\lambda_\ell^2$ for all $\ell$, 
only the D\"orfler marking is applied.

\begin{figure}
 \begin{center}
 \includegraphics{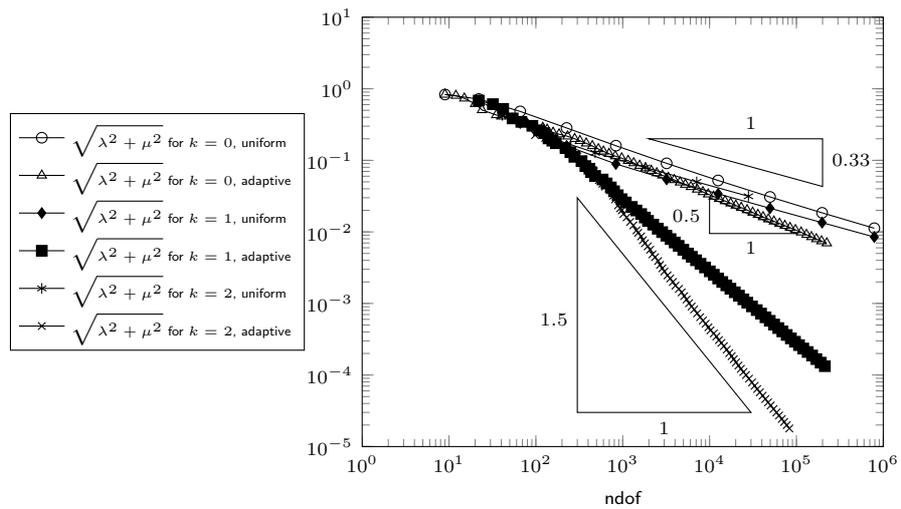}
 \end{center}
 \caption[Error estimators for 
 the experiment from Subsection~\ref{ss:PMPnumLshaped}.]{\label{f:PMPnumLshaped}Error estimators for 
 the experiment from Subsection~\ref{ss:PMPnumLshaped}.}
\end{figure}

\begin{figure}
 \begin{center}
  \includegraphics[width=0.32\textwidth]
      {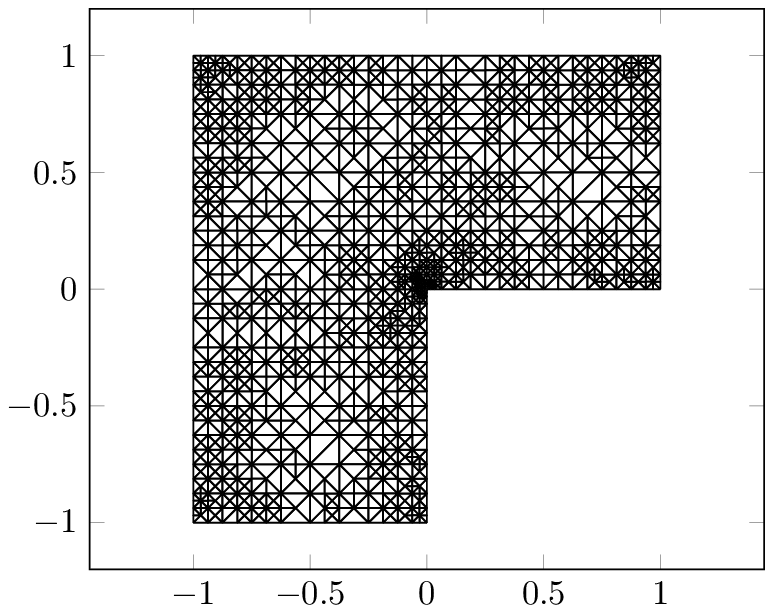}
  \includegraphics[width=0.32\textwidth]
      {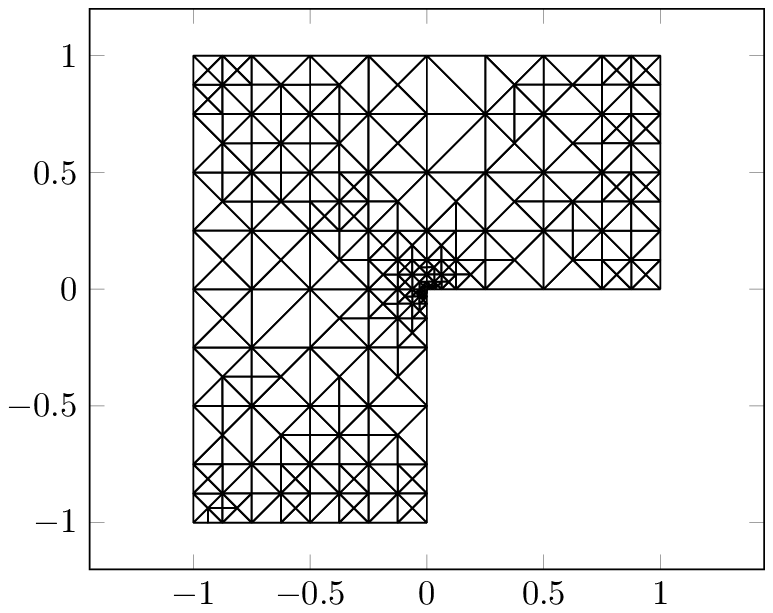} 
  \includegraphics[width=0.32\textwidth]
      {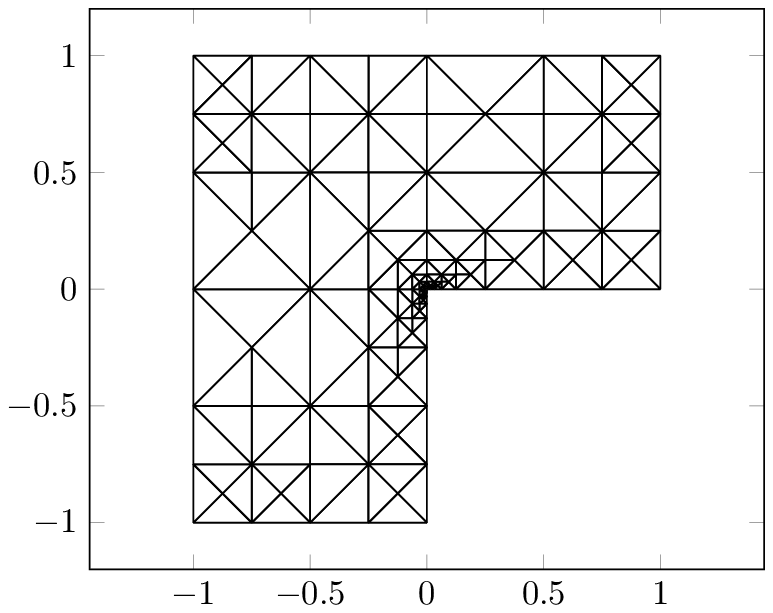}
 \end{center}
  \caption[Adaptively refined triangulations 
 for the experiment from 
Subsection~\ref{ss:PMPnumLshaped}.]{\label{f:PMPnumLshapedTriang}Adaptively 
refined triangulations for the experiment from 
Subsection~\ref{ss:PMPnumLshaped}.}
\end{figure}

\subsection{Singular $\alpha$}\label{ss:PMPnumSingAlpha}

This subsection is devoted to a numerical investigation of the dependence of 
the error $\|p-p_h\|_{L^2(\Omega)}$ on the regularity of $\alpha$.
The exact 
smooth solution $u\in C^\infty(\Omega)$ of 
\begin{align*}
 -\Delta u= 2\sin(\pi x)\sin(\pi y)\text{ in }\Omega
 \qquad\text{and}\qquad
  u\vert_{\Gamma_D}=0
\end{align*}
reads $u(x,y) = \sin(\pi x)\sin(\pi y)$.
Define $\varphi = \nabla u + \Curl(\widetilde{\alpha})$
with $\widetilde{\alpha}\in H^1(\Omega)\setminus H^2(\Omega)$ defined by
$\widetilde{\alpha}(r,\phi) = r^{2/3} \sin(2\phi/3)$.
Then $\varphi\in H(\ddiv,\Omega)$ with $-\ddiv\varphi=f$.

The errors and error estimators are plotted in 
Figure~\ref{f:PMPnumLshapedSingAlpha} against the number of degrees of freedom. 
The convergence rate on uniform red-refined meshes for $k=1,2$ is 
$h^{2/3}\approx\texttt{ndof}^{-1/3}$ and, hence, the convergence rate seems 
to depend on the regularity of $\alpha$. 
The errors and error estimators show the same convergence rate. 
Figure~\ref{f:PMPnumLshapedSingAlpha_P1} focuses on the results for $k=0$ and 
uniform mesh-refinement. 
The error $\|p-p_h\|_{L^2(\Omega)}$ and the error estimator $\sqrt{\lambda^2+\mu^2}$ show a 
convergence rate between $h$ and $h^{2/3}$, while 
$\left\|\Curl(\alpha-\alpha_h)\right\|_{L^2(\Omega)}$ converges with a rate of 
$h^{2/3}\approx\texttt{ndof}^{-1/3}$ due to the singularity of $\alpha$.
This numerical experiment suggests that the error $\|p-p_h\|_{L^2(\Omega)}$
does not depend on the regularity of $\alpha$ 
(at least in a preasymptotic regime).
The triangle inequality implies 
$\left\|\Curl(\alpha-\alpha_h)\right\|_{L^2(\Omega)}\leq 
\|p-p_h\|_{L^2(\Omega)}+\mu$. This upper bound is also plotted in 
Figure~\ref{f:PMPnumLshapedSingAlpha_P1}.

Figure~\ref{f:PMPnumLshapedSingCurlTriang} depicts adaptively refined meshes 
for $k=0,1,2$ with approximately 1000 degrees of freedom. The singularity of 
$\alpha$ leads to a strong refinement towards the re-entrant corner.
The marking with respect to the data-approximation 
($\mu_\ell^2>\kappa\lambda_\ell^2$ in Algorithm~\ref{a:PMPafem}) is only 
applied at levels 1--5, 7, 12, and 18 for $k=0$. All other marking steps for 
$k=0,1,2$ use the D\"orfler marking ($\mu_\ell^2\leq\kappa\lambda_\ell^2$).

\begin{figure}
 \begin{center}
 \includegraphics{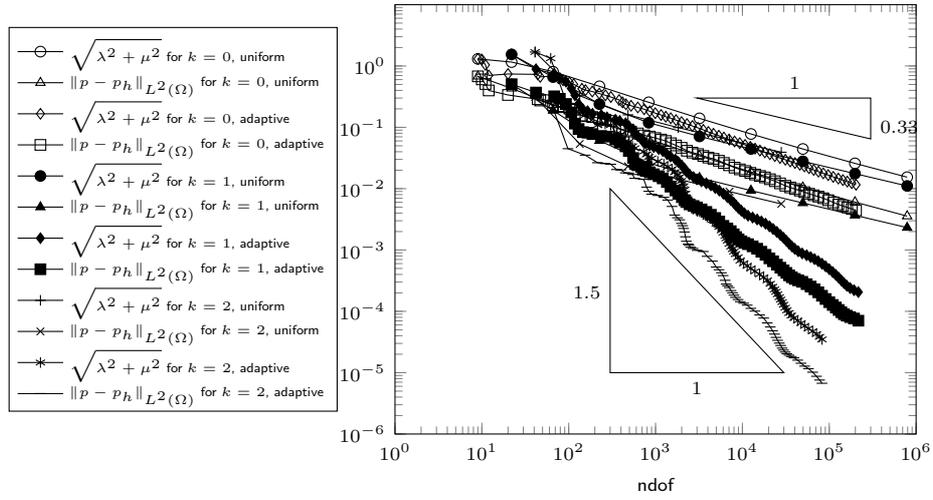}
 \end{center}
 \caption[Errors and error estimators for 
 the experiment
 from Subsection~\ref{ss:PMPnumSingAlpha}.]{\label{f:PMPnumLshapedSingAlpha}Errors 
 and error estimators for 
 the experiment with singular $\alpha$ 
 from Subsection~\ref{ss:PMPnumSingAlpha}.}
\end{figure}

\begin{figure}
 \begin{center}
 \includegraphics{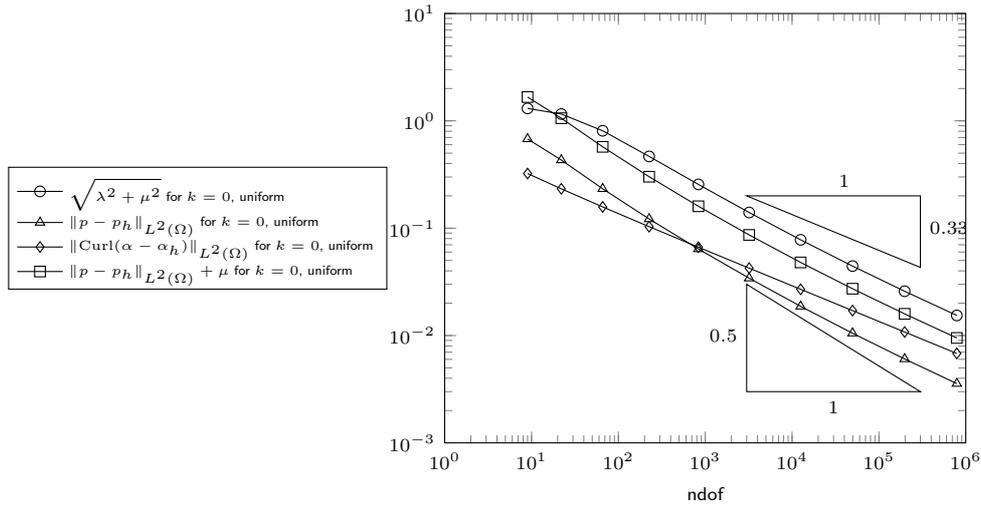}
 \end{center}
 \caption[Errors and error estimators for 
 the experiment from Subsection~\ref{ss:PMPnumSingAlpha}
 and uniform refinement.]{\label{f:PMPnumLshapedSingAlpha_P1}Errors and 
 error estimators for 
 the experiment with singular $\alpha$ 
 from Subsection~\ref{ss:PMPnumSingAlpha} and uniform refinement.}
\end{figure}

\begin{figure}
 \begin{center}
  \includegraphics[width=0.32\textwidth]
      {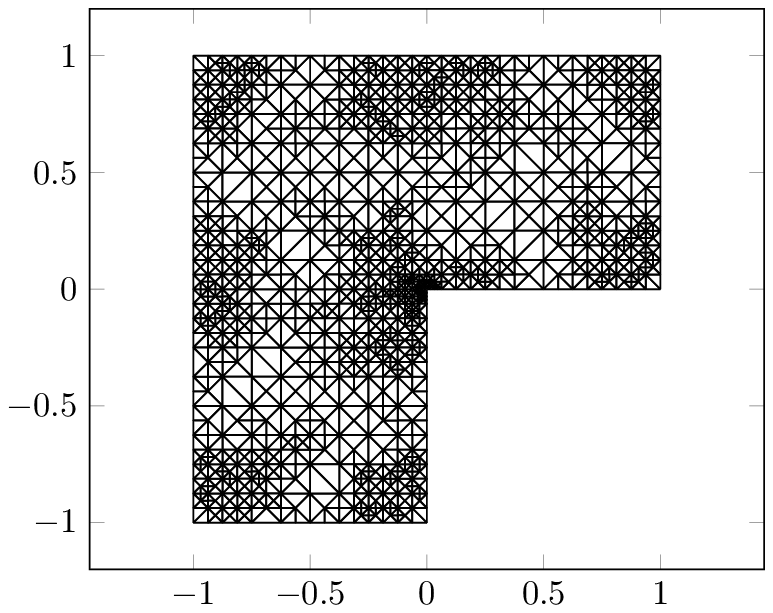}
  \includegraphics[width=0.32\textwidth]
      {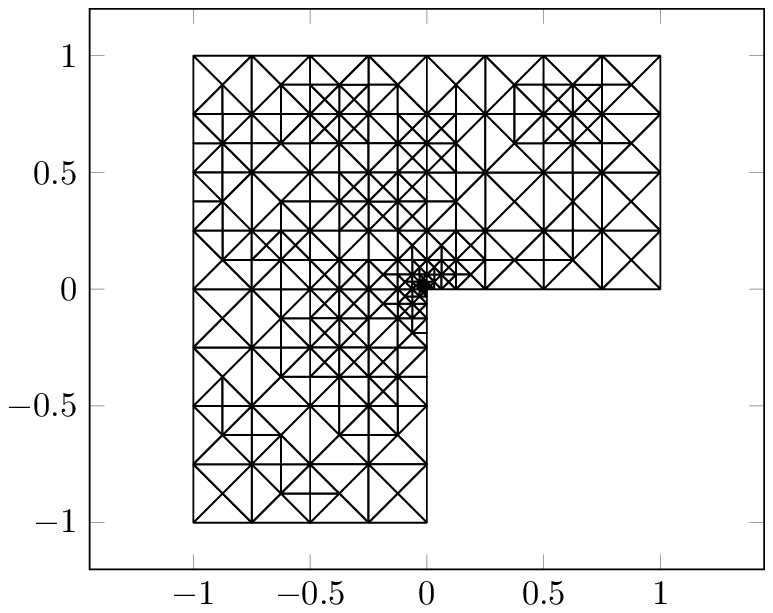} 
  \includegraphics[width=0.32\textwidth]
      {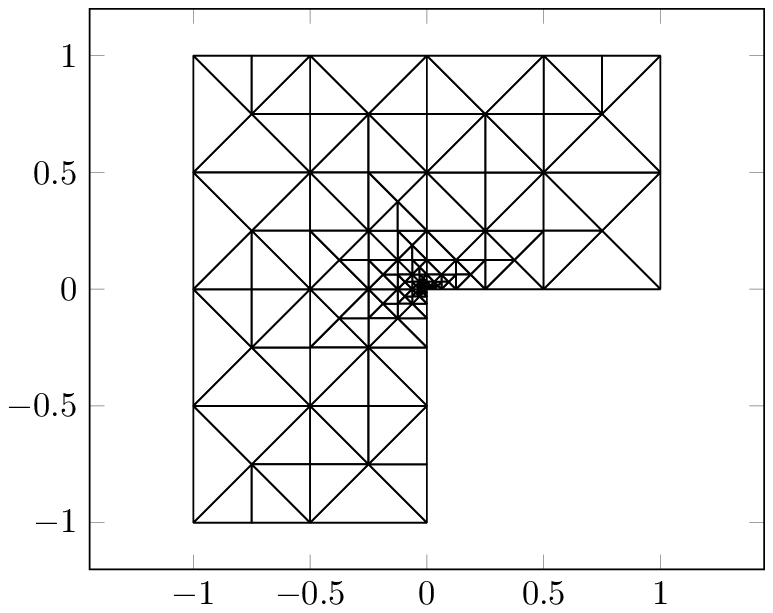}
 \end{center}
  \caption[Adaptively refined triangulations 
 for the experiment from 
Subsection~\ref{ss:PMPnumSingAlpha}.]{\label{f:PMPnumLshapedSingCurlTriang}Adaptively refined 
triangulations for the experiment from 
Subsection~\ref{ss:PMPnumSingAlpha}.}
\end{figure}

\section*{Acknowledgement}
The author would like to thank Professor C.\ Carstensen for valuable 
discussions.

\end{document}